\documentclass[12pt,a4paper]{amsart}

\usepackage{amsmath}
\usepackage{amscd}
\usepackage{xypic}  
\usepackage{amssymb}
\usepackage{amsthm}
\usepackage{epsfig}
\usepackage{paralist}
\usepackage{enumerate}
\usepackage[normalem]{ulem} 
\usepackage{hyperref}
\usepackage[utf8]{inputenc}

\usepackage{comment}

\usepackage[all,cmtip]{xy}
\usepackage{amsmath}
\usepackage{color}

\theoremstyle{plain}
\newtheorem{theorem}{Theorem}[section]
\newtheorem{thm}{Theorem}[section] 

\newtheorem{lemma}[thm]{Lemma}

\newtheorem{proposition}[thm]{Proposition}

\newtheorem{claim}[thm]{Claim}

\newtheorem{corollary}[thm]{Corollary}

\theoremstyle{definition}

\newtheorem{remark}[thm]{Remark}
\newtheorem{remarks}[thm]{Remarks}

  \newtheorem{definition-remark}[thm]{Definition-Remark}
  \newtheorem{example}[thm]{Example}
 \newtheorem{sit}[theorem]{}

\def\codim{\operatorname{codim}}

\def\max{\operatorname{max}}

\def\c1{\operatorname{c_1}}
\def\c2{\operatorname{c_2}}

\def\Sym{\operatorname{Sym}}

\def\CC{{\mathbb C}}

\def\ZZ{{\mathbb Z}}
\def\NN{{\mathbb N}}

\def\PP{{\mathbb P}}

\def\SS{{\mathcal S}}

\def\O{{\mathcal O}}

\def\c{\mathfrak{c}}

\def\cong{\simeq}

\def\+{\oplus}               
\def\*{\otimes}                  

\begin{document}

\title[Fano schemes of complete intersections]{
On Fano schemes of complete intersections}

\author{C.~Ciliberto, 
M.~Zaidenberg}
\address{Dipartimento di Matematica, Universit\`a degli
Studi di Roma ``Tor Vergata'', Via della Ricerca Scientifica,
00133 Roma, Italy
}
\email{cilibert@mat.uniroma2.it}
\address{Univ.\ Grenoble Alpes, CNRS, Institut Fourier, F-38000 Grenoble, France}
\email{Mikhail.Zaidenberg@univ-grenoble-alpes.fr}


\begin{abstract}  We provide enumerative formulas for the degrees of varieties parameterizing hypersurfaces and complete intersections which contain projective subspaces and  conics. Besides, we find all cases where the Fano scheme of the general complete intersection is irregular of dimension at least 2, and for the Fano surfaces we deduce formulas for their holomorphic Euler characteristic.\end{abstract}

\maketitle

\tableofcontents

\section*{Introduction}

The study of hypersurfaces in projective space, or more generally, of complete intersection, and specifically of varieties contained in them, is a classical subject in algebraic geometry. The present paper is devoted to this subject, and in particular to some enumerative aspects of it. 

Recall that the {\em Fano scheme} $F_k(X)$ of a projective variety $X\subset\PP^r$ is the Hilbert scheme of $k$-planes (that is, linear subspaces of dimension $k$) contained in $X$; see \cite{AK} or \cite[14.7.13]{Ful}. For a hypersurface $X\subset\PP^r$ of degree $d$ the integer
$$\delta(d,r,k)=(k+1)(r-k)-{{d+k}\choose k}$$ is called the 
{\em expected dimension} of $F_k(X)$. Let $\Sigma(d,r)$ be the projective space of dimension ${{d+r}\choose r}-1$ which parameterizes the hypersurfaces of degree $d\ge 3$ in $\PP^ r$.
If either $\delta(d,r,k)<0$ or $2k\ge r$ then $F_k(X)=\emptyset$ for a general 
$X\in\Sigma(d,r)$. Otherwise $F_k(X)$ has dimension $\delta$ for a general 
$X\in\Sigma(d,r)$ (\cite {B, DM, L, Pr}). 
Setting
\begin{equation*}\label{eq:gamma}
\gamma(d,r,k)=-\delta(d,r,k)>0\,
\end{equation*}
the general hypersurface of degree $d\ge 3$ in $\PP^ r$ contains no $k$-plane. 
Let $\Sigma(d,r,k)$ be the subvariety of $\Sigma(d,r)$ of points corresponding to hypersurfaces which carry $k$-planes. Then $\Sigma(d,r,k)$ is a nonempty, irreducible, proper subvariety of codimension $\gamma(d,r,k)$ in $\Sigma(d,r)$ (see \cite{Mi}), and its general point corresponds to a hypersurface of degree $d$ which carries a unique $k$-plane (see \cite{BCFS}). The degree of this subvariety of the projective space $\Sigma(d,r)$ was computed in \cite{Man0}. In Section \ref{sec:hyper} we reproduce this computation. 
This degree $\deg(\Sigma(d,r,k))$ is the total number of $k$-planes in the members of the general linear system $\mathcal L$ of degree $d$ hypersurfaces, provided $\dim(\mathcal L)=\gamma(d,r,k)$. It can be interpreted also as the top Chern number of a vector bundle. Having in mind the further usage, we explore three different techniques for computing it: \begin{itemize}\item the Schubert calculus; \item a trick due to Debarre-Manivel; \item the Bott residue formula and the localization in the equivariant Chow rings. \end{itemize}
In Section 2 we extend these computations to the Fano schemes of complete intersections in $\PP^r$. 

In Sections \ref{sec:num-inv}-\ref{sec:irreg-Fano} we turn to the opposite case $\gamma(d,r,k)<0$, that is, the expected dimension of the Fano scheme  is positive. In Section  \ref{sec:num-inv} we compute certain Chern classes related to the Fano scheme. In Section \ref{sec:Fano-surf} we apply these computations in the case where the Fano scheme is a surface, and provide several concrete examples. The main result of Section \ref{sec:irreg-Fano}  describes all the cases where the Fano scheme of the general complete intersection has dimension $\ge 2$ and a positive irregularity. 
This happens only for the general cubic threefolds in $\PP^4$ ($k=1$), 
the general cubic fivefolds in $\PP^6$ ($k=2$), and the general intersections of two quadrics in $\PP^{2k+3}$, $k\ge 1$; see Theorem \ref{thm:irregular}.

In the final Section \ref{sec:conics} we turn to the conics in degree $d$ hypersurfaces in $\PP^r$. Let \[\epsilon(d,r)=2d+2-3r\,.\] Let $\Sigma_c(d,r)$ be the subvariety of $\Sigma(d,r)$ consisting of the degree $d$ hypersurfaces which contain conics. 
We show that $\Sigma_c(d,r)$ is irreducible of codimension  $\epsilon(d,r)$ in $\Sigma(d,r)$, provided $\epsilon(d,r)\geq 0$. Then we prove that the general hypersurface in $\Sigma_c(d,r)$ contains a unique (smooth) conic if $\epsilon(d,r)> 0$.
Our main results in this section are  formulas \eqref{eq:deg-conics}-\eqref{eq:deg-conics1} which express the degree of $\Sigma_c(d,r)$ via Bott's residue formula.  Notice that there exists already a formula for $\deg(\Sigma_c(d,r))$ in the case $r=3$,  $d\ge 5$, that is, for the surfaces in $\PP^3$,  see \cite[Prop.\ 7.1]{MaRoXaVa}. It expresses this degree as a polynomial in $d$. 

Let us finish with a few comments on the case $\epsilon(d,r)< 0$.
It is known (see \cite{HRS1})
that for $2d \le r+1$, given a general hypersurface $X\subset\PP^r$ of degree $d$ and any point $x\in X$, there is a family of dimension $e(r+1 - d) - 2\ge ed$
of degree $e$ rational curves containing $x$. In particular, $X$ carries a $2(r-d)$-dimensional family of smooth conics through an arbitrary point. Moreover (see \cite{BK}), for $3d\le 2r-1$ the Hilbert scheme of smooth rational curves of degree $e$ on a general $X$ is irreducible of the \emph{expected dimension} $e(r- d + 1 ) + r - 4$. In particular, 
the Hilbert scheme of smooth conics in $X$ is  irreducible of dimension $3r-2d-2=-\epsilon(d,r)$. Analogs of the latter statements hold as well for general complete intersections (see \cite{BK}). See also \cite{Bea1, BH} for enumerative formulas counting conics in complete intersections. 

\medskip

{\bf Acknowledgments}: The first author acknowledges the MIUR Excellence Department Project awarded to the Department of Mathematics, University of Rome Tor Vergata, CUP E83C18000100006. He also thanks the GNSAGA of INdAM. This research was  partially done during a visit of the second author at the Department of Mathematics, University of Rome Tor Vergata (supported by the project "Families of curves: their moduli and their 
related varieties", CUP E8118000100005, in the framework of Mission 
Sustainability).
The work of the second author was also partially supported by the grant 346300 for IMPAN from the Simons Foundation and the matching 2015-2019 Polish MNiSW fund (code: BCSim-2018-s09). The authors thank all these Institutions and programs for the support and excellent working conditions.

Our special thanks are due to Laurent Manivel who kindly suggested to complete some references omitted in the first draft of this paper. 

\section{Hypersurfaces containing linear subspaces}\label{sec:hyper}

The results of this section are known, see \cite{Man0}, except maybe for formula \eqref{eq:Bott-1}. Our aim is rather didactic, we introduce here the techniques that will be explored in the subsequent sections.

Recall that the {\em Fano scheme} $F_k(X)$ of a projective variety $X\subset\PP^r$ is the Hilbert scheme of linear subspaces of dimension $k$ contained in $X$; see \cite{AK} or \cite[14.7.13]{Ful}. For a hypersurface $X\subset\PP^r$ of degree $d$ the integer
$$\delta(d,r,k)=(k+1)(r-k)-{{d+k}\choose k}$$ is called the 
{\em expected dimension} of $F_k(X)$. Let $\Sigma(d,r)$ be the projective space of dimension ${{d+r}\choose r}-1$ which parameterizes the hypersurfaces of degree $d\ge 3$ in $\PP^ r$.
If either $\delta(d,r,k)<0$ or $2k\ge r$ then $F_k(X)=\emptyset$ for a general 
$X\in\Sigma(d,r)$. Otherwise $F_k(X)$ has dimension $\delta$ for a general 
$X\in\Sigma(d,r)$ (\cite {B, DM, L, Pr}). 
We assume in the sequel that
\begin{equation}\label{eq:gamma}
\gamma(d,r,k):=-\delta(d,r,k)>0\,.
\end{equation}
Then the general hypersurface of degree $d\ge 3$ in $\PP^ r$ contains no linear subspace of dimension $k$. 
Let $\Sigma(d,r,k)$ be the subvariety of $\Sigma(d,r)$ of points corresponding to hypersurfaces which do contain a linear subspace of dimension $k$. 
The following statement, proven first in \cite[Thm.\ (1)-(2)]{Man0} in a slightly weaker form, is a particular case of Theorem 1.1 in \cite{BCFS}; see Proposition \ref {prop:ci} below.

\begin{proposition}\label{prop:tbp} Assume $\gamma(d,r,k)>0$. 
Then $\Sigma(d,r,k)$ is a nonempty, irreducible and rational subvariety of codimension  $\gamma(d,r,k)$ in $\Sigma(d,r)$. The general point of $\Sigma(d,r,k)$ corresponds to a hypersurface which contains a unique linear subspace of dimension $k$ and has singular locus of dimension $\max\{-1, 2k-r\}$ along its unique $k$-dimensional linear subspace (in particular, it is smooth provided $2k<r$).
\end{proposition} 

For instance, take $d=3$, $r=5$, and $k=2$. Then $\Sigma(3,5)$ parameterizes the  cubic fourfolds in $\PP^5$, and $\Sigma(3,5,2)$ parameterizes those cubic fourfolds which contain a plane. Since $\gamma(3,5,2)=1$, we conclude that $\Sigma(3,5,2)$ is a divisor in $\Sigma(3,5)$, and the general point of $\Sigma(3,5,2)$ corresponds to a smooth cubic fourfold which contains a unique plane. 

Our aim is to compute the degree of $\Sigma(d,r,k)$ in the projective space $\Sigma(d,r)$ in the case $\gamma(d,r,k)>0$. 

\begin{sit}\label{sit:S-star} On the Grassmannian $\mathbb G(k,r)$ of $k$--subspaces of $\PP^ r$, consider the dual $\SS^ *$ of the tautological vector bundle $\SS$ of rank $k+1$.  Let $\Pi\in\mathbb G(k,r)$ correspond to a $k$-subspace of $\PP^r$. Then the fiber of $\SS^ *$ over $\Pi$   is $H^ 0(\Pi,\O_\Pi(1))$. It is known (\cite [Sect.\ 5.6.2]{EH}, \cite  {Ful}) that
\[
c(\SS^ *)=1+\sum _{i=1}^ {k+1} \sigma _{(1^ i)},
\]
 where $(1^ i)$ stays for the vector $(1,\ldots,1)$ of length $i$, and  $\sigma _{(1^ i)}$ is the (Poincar\'e dual of the) class of the Schubert cycle $\Sigma _{(1^ i)}$. This cycle has codimension $i$ in $\mathbb G(k,r)$, therefore, $\Sigma _{(1^ i)}\in A^i(\mathbb G(k,r))$ in the Chow ring $A^*(\mathbb G(k,r))$. 

The \emph{splitting principle} (see \cite[Sect.\ 5.4]{EH}) says that any relation among Chern classes which holds for all split vector bundles holds as well for any vector bundle. So, we can write formally
\[
\SS^ *=L_0\oplus \ldots \oplus L_k\,,
\]
the $L_i$s being (virtual) line bundles. In terms of the \emph{Chern roots} $x_i=c_1(L_i)$ one can express
\[
c(\SS^ *)=1+c_1(\SS^ *)+\ldots+c_{k+1}(\SS^ *)=(1+x_0)\cdots (1+x_k)\,.
\] 
Hence $\sigma _{(1^ i)}$ is the $i$--th  elementary symmetric polynomial in $x_0,\ldots, x_k$, i.e.,
\[
\sigma_{(1)}=x_0+\ldots+x_k,\;\; \sigma_{(1^2)}=\sum_{0\leqslant i<j\leqslant k} x_ix_j, \;\; \ldots, \;\;  \sigma_{(1^{k+1})}=x_0\ldots x_k.
\]

Consider further the vector bundle ${\rm Sym}^ d(\SS^*)$ on $\mathbb G(k,r)$ of rank $${{d+k}\choose k}>(k+1)(r-k)=\dim (\mathbb G(k,r)).$$ To compute the Chern class of ${\rm Sym}^ d(\SS^*)$ one writes
\[
{\rm Sym}^ d(\SS^*)=\bigoplus_{v_0+\ldots+v_k=d}L_0^{v_0}\cdots L_k^{v_k}.
\]
Since $c_1(L_0^{v_0}\cdots L_k^{v_k})=v_0x_0+\ldots+v_kx_k$ one obtains
\begin{equation}\label{eq:ch}
c({\rm Sym}^d(\SS^*))=\prod_{v_0+\ldots+v_k=d}(1+v_0x_0+\ldots+v_kx_k)\,.
\end{equation}
\end{sit}
  
The following lemma is standard, see, e.g., \cite[Thm.\ (3)]{Man0}. 
For the reader's convenience we include the proof. As usual, the integral of the top degree cohomology class stands for its value on the fundamental cycle. 
The integral of the dual of a zero cycle $\alpha$ coincides with the degree of $\alpha$.

  \begin{lemma}\label{lem:Chern} Suppose \eqref{eq:gamma} holds. 
   Then 
  one has
\[
\deg (\Sigma(d,r,k))=\int_{\mathbb G(k,r)} c_{(k+1)(r-k)}({\rm Sym}^ d(\SS^*)).
\]
  \end{lemma}
  
  \begin{proof} 
 Let $p: V(k,r)\to \mathbb G(k,r)$ be the tautological $\PP^k$-bundle over the Grassmannian $\mathbb G(k,r)$. Consider the composition 
 $$\varphi: V(k,r)\stackrel{\phi}{\hookrightarrow} \PP^r\times\mathbb G(k,r)\stackrel{\pi}{\longrightarrow} \PP^r\,,$$ 
 where $\phi$ is the natural embedding  and $\pi$ stands for the projection  to the first  factor. Letting $\mathcal{T}=\O_{\PP^r}(-1)$ 
 and $\mathcal{S}_d={\rm Sym}^d (\mathcal{T}^*)$ one obtains
$$ \mathcal{S}^*=R^0p_*\varphi^*(\mathcal{T}^*)\quad\mbox{and}\quad {\rm Sym}^d (\mathcal{S}^*)=R^0p_*\varphi^*(\mathcal{S}_d)\,.$$ 
Any $F\in H^0(\PP^r, \O_{\PP^r}(d))$ defines a section $\sigma_F$ of 
${\rm Sym}^d (\mathcal{S}^*)$ such that $\sigma_F(\Pi)=F|_\Pi\in H^0(\Pi,\O_\Pi(d))$. Consider the hypersurface $X_F$ of degree $d$  on $\PP^r$ with equation $F=0$. The support of $X_F$ contains a linear subspace $\Pi\in\mathbb G(k,r)$ if and only if $\sigma_F(\Pi)=0$, i.e., the subspaces $\Pi\in\mathbb G(k,r)$ lying in Supp$(X_F)$ correspond to the zeros of $\sigma_F$ in $\mathbb G(k,r)$, which have a natural scheme structure.
 
 Let  $\rho=\dim(\mathbb G(k,r))=(k+1)(r-k)$. By our assumption one has  $${\rm rk}\,({\rm Sym}^d (\mathcal{S}^*))-\rho=\gamma(r,k,d)>0\,.$$
Choose a general linear subsystem $\mathcal{L}=\langle X_0,\ldots,X_\gamma\rangle$ in $\Sigma(d,r)=|\O_{\PP^r}(d)|$ of dimension $\gamma=\gamma(r,k,d)$, where $X_i=\{F_i=0\}$. By virtue of Proposition \ref{prop:tbp}, $\mathcal{L}$ meets $\Sigma (d,r,k)\subset\Sigma(d,r)$ transversally in $\deg (\Sigma(d,r,k))$ simple points, and to any such point $X\in\Sigma (d,r,k)\cap\mathcal{L}$ corresponds a unique $k$-dimensional subspace $\Pi\in \mathbb G(k,r)$ such that $\Pi\subset X$.

Consider now  the sections $\sigma_i:=\sigma_{F_i}$, $i=0,\ldots,\gamma$, of ${\rm Sym}^d (\mathcal{S}^*)$. The intersection of $\mathcal L$ with $\Sigma(d,r,k)$ is exactly the scheme of points $\Pi\in \mathbb G(k,r)$ where there is a linear combination of $\sigma_0,\ldots, \sigma_\gamma$ vanishing on $\Pi$.  This is the zero dimensional scheme of points of $\mathbb G(k,r)$ where the sections $\sigma_0,\ldots, \sigma_\gamma$ are linearly dependent. This zero dimensional scheme represents the top Chern class $c_\rho({\rm Sym}^d (\mathcal{S}^*))$ (see \cite [Thm. 5.3]{EH}).
Its degree (which is equal to $\deg (\Sigma(d,r,k))$) is the required Chern number $\int_{\mathbb G(k,r)} c_{(k+1)(r-k)}({\rm Sym}^ d(\SS^*))$. 
 \end{proof}

Let us  explain now three methods for computing $\deg (\Sigma(d,r,k))$. \medskip

\subsection {Schubert calculus} In order to compute $c_{(r-k)(k+1)}({\rm Sym}^ d(\SS^*))$, one computes the polynomial in $x_0,\ldots, x_k$ appearing in \eqref {eq:ch} and extracts the homogeneous component $\tau_{(d,r,k)}$ of degree $(k+1)(r-k)$. The latter homogeneous polynomial  in $x_0,\ldots, x_k$ is symmetric, hence it can be expressed via a polynomial in the elementary symmetric functions $\sigma_{(1^ i)}$, $i=0,\ldots,k+1$:
\begin{equation}\label{eq:tau}
\tau_{(d,r,k)}=\sum_{j_1+2j_2+\ldots+(k+1)j_{k+1}=(k+1)(r-k)} \phi_{d,r}(j_1,j_2,\ldots, j_{k+1}) \sigma_{(1)}^ {j_1}\sigma_{(1^ 2)}^ {j_2}\cdots 
\sigma_{(1^ {k+1})}^ {j_{k+1}}
\end{equation}
with suitable coefficients $\phi_{d,r}(j_1,j_2,\ldots, j_{k+1})$. In this way the top Chern number $$\tau_{(d,r,k)}=c_{(k+1)(r-k)}\left({\rm Sym}^ d(\SS^*)\right)$$ in Lemma \ref{lem:Chern} is expressed in terms of the Chern numbers $$\sigma_{(1)}^ {j_1}\sigma_{(1^ 2)}^ {j_2}\cdots 
\sigma_{(1^ {k+1})}^ {j_{k+1}}$$ appearing in \eqref {eq:tau}. By computing the intersection products among Schubert classes and plugging these in \eqref{eq:tau} one obtains the desired degree
\begin{equation*}\label{eq:tau0}\deg (\Sigma(d,r,k))=\tau_{(d,r,k)}.\end{equation*}

\subsection{Debarre--Manivel's trick} \label{ssec:DM} This trick (applied for a similar purpose by van der Waerden \cite{vdW}) allows to avoid passing to the elementary symmetric polynomials, which requires to compute the coefficients in \eqref{eq:tau}.  Let us recall the basics on the Chow ring of the Grassmannian $\mathbb G(k, r)$ following \cite{Man}. 

A \emph{partition} ${\bf\lambda}$ of length $k+1$ is a (non-strictly) decreasing sequence of non-negative integers $(\lambda_0,\ldots,\lambda_k)$. 
To such a partition ${\bf\lambda}$ there corresponds 
a homogeneous symmetric \emph{Schur polynomial} $s_{\bf\lambda}\in\ZZ[x_0,\ldots,x_k]$ of degree $|{\bf\lambda}|=\lambda_0+\ldots+\lambda_k$. These polynomials form a base of the $\ZZ$-module $\Lambda_{k+1}$ of symmetric polynomials in $x_0,\ldots,x_k$. 
One wrights ${\bf\lambda}\subset (k+1)\times(r-k)$ if
$r-k\ge \lambda_0\ge\ldots\ge\lambda_k\ge 0$. This inclusion means that the corresponding \emph{Ferrers diagram} of ${\bf\lambda}$ is inscribed in the  rectangular matrix of size $(k+1)\times(r-k)$ occupying $\lambda_{i-1}$ first places of the $i$th line for $i=1,\ldots,k+1$. To any ${\bf\lambda}\subset (k+1)\times(r-k)$ there correspond:
\begin{itemize}
\item a Schubert variety $\Sigma_{\bf\lambda}\subset\mathbb G(k, r)$ of codimention $|{\bf\lambda}|$;
\item the corresponding Schubert cycle $[\Sigma_{\bf\lambda}]$ in the Chow group $A_*(\mathbb G(k, r))$;
\item the corresponding dual Schubert class $\sigma_{\bf\lambda}$ of degree $|{\bf\lambda}|$ in the Chow ring $A^{*}(\mathbb G(k, r))$. 
\end{itemize} 
The nonzero Schubert classes form a base of the  free $\ZZ$-module $A^{*}(\mathbb G(k, r))$  (\cite[Cor.\ 3.2.4]{Man}). There is a unique partition ${\bf\lambda_{\rm max}}=(r-k,\ldots,r-k)\subset (k+1)\times(r-k)$ of maximal weight $|{\bf\lambda}_{\rm max}|=(k+1)(r-k)$. Its Ferrers diagram coincides with the whole rectangle $(k+1)\times(r-k)$.  The corresponding Schur polynomial is $s_{\bf\lambda_{\rm max}}=(x_0\cdots x_k)^{r-k}$. The corresponding Schubert cycle is a reduced point, and the corresponding Schubert class $\sigma_{\bf\lambda_{\rm max}}$ generates the $\ZZ$-module $A^{(k+1)(r-k)}(\mathbb G(k, r))\simeq\ZZ$.  

Let as before $x_0,\ldots,x_k\in A^{1}(\mathbb G(k, r))$ be the Chern roots of the vector bundle $\SS^ *$ over $\mathbb G(k, r)$. A homogeneous symmetric polynomial $\tau\in\ZZ[x_0,\ldots,x_k]$ of degree $(k+1)(r-k)$ admits a unique decomposition as an integral linear combination of Schur polynomials $s_{\bf\lambda}$ of the same degree.  The corresponding Schubert classes $\sigma_{\bf\lambda}$ vanish except for  $\sigma_{\bf\lambda_{\rm max}}$.  If $\tau$  corresponds to an effective zero cycle on $\mathbb G(k, r)$, then the degree of this cycle equals the coefficient of $s_{{\bf\lambda}_{\rm max}}=(x_0\cdots x_k)^{r-k}$ in the decomposition of $\tau$ as a linear combination of Schur polynomials. Multiplying $\tau$ by the Vandermonde polynomial
\[
V=V(x_0,\ldots, x_k)=\prod_{0\leqslant i<j\leqslant k} (x_i-x_j)\,,
\] this coefficient becomes the coefficient of the monomial $x_0^ rx_1^ {r-1}\cdots x_k^ {r-k}$ in the product $\tau\cdot V$, see the proofs of \cite [Thm. 4.3]{DM} and \cite[Thm.\ 3.5.18]{Man}. 
 
Let $P(x_0,\ldots, x_k)$ be a polynomial, and let $x_0^ {i_0}\cdots x_k^ {i_k}$ be a monomial, which we identify with the lattice vector ${\bf i}=(i_0,\ldots, i_k)\in\ZZ^{k+1}$.  We write
$\psi_{\bf i}(P)$ for the coefficient of $x_0^ {i_0}\cdots x_k^ {i_k}$ in  $P=\sum_{\bf i}\psi_{\bf i}(P)x_0^ {i_0}\cdots x_k^ {i_k}$. 
Summarizing the preceding discussion and taking into account Lemma \ref {lem:Chern} one arrives at the following conclusion.

\begin{proposition}\label{prop:DB-trick}  {\rm (\cite[pp.\ 311-312]{Man0})}
One has 
\[
\deg (\Sigma(d,r,k))=\psi_{(r,r-1,\ldots,r-k)}(V\cdot\tau_{(d,r,k)})\,,
\]
that is, the degree
$\deg (\Sigma(d,r,k))$ equals the coefficient of  $x_0^ rx_1^ {r-1}\cdots x_k^ {r-k}$ in the product $\tau\cdot V$, where $\tau=\tau_{(d,r,k)}(x_0,\ldots, x_k)$ is as in \eqref{eq:tau}.
\end{proposition}

 See also the table in \cite[p.\ 312]{Man0} of the values of $\deg (\Sigma(d,r,k))$ for several different values of $(d,r,k)$, and an explicit formula in \cite[Cor.\ on p.\ 312]{Man0} which expresses $\deg (\Sigma(d,3,1))$ as a polynomial in $d$ of degree $8$. 

\subsection{Bott's residue formula}\label{ssec:bott}  Bott's residue formula \cite[Thms.\ 1, 2]{Bo} says, in particular, that one can compute the degree of a zero--dimensional cycle class on a smooth projective variety $X$ in terms of local contributions given by the fixed point loci of a torus action on $X$. Here we follow the treatment in \cite {H1} based on \cite {Br}, \cite {EG}, and \cite {MV} and adapted to our setting. 

We consider the diagonal action of  $T={(\mathbb C^ *)}^ {r+1}$ on $\PP^ r$ given in coordinates by
\[
(t_0,\ldots,t_r)\cdot (x_0:\ldots:x_r)= (t_0x_0:\ldots : t_rx_r).
\]
This induces an action of $T$ on $\mathbb G(k,r)$, with ${r+1}\choose {k+1}$ isolated fixed points corresponding to the coordinate $k$--subspaces in  $\PP^ r$, which are indexed by the subsets $I$ of order $k+1$ of the set $\{0,\ldots, r\}$. We let $\mathcal I_{k+1}$ denote the set of all these subsets, and  $\Pi_I\in\mathbb G(k,r)$ denote the subspace  which corresponds to  $I\in \mathcal I_{k+1}$.
Bott's residue formula, applied in our setting, has the form
\[
\deg (\Sigma(d,r,k))=\int_{\mathbb G(k,r)} c_{(r-k)(k+1)}({\rm Sym}^ d(\SS^*))=\sum_{I\in \mathcal I_{k+1}} \frac {c_I}{e_I}\,,
\]
where $c_I$ results from the local contribution of 
$c_{(r-k)(k+1)}({\rm Sym}^ d(\SS^*))$ at $\Pi_I$, and  $e_I$ is determined by the torus action on the tangent space to $\mathbb G(k,r)$ at  
$\Pi_I$. 

As for the computation of $e_I$, this goes exactly as in \cite [p. 116]{H1}, namely
\[
e_I=(-1)^ {(k+1)(r-k)} \prod_{i\in I}\prod_{j\not\in I}(t_i-t_j).
\]
Also the computation of $c_I$ is similar to the one made in  \cite [p. 116]{H1}. Recalling  \eqref {eq:ch}, for a given $I\in \mathcal I_{k+1}$, consider the polynomial
\[
\prod_{v_0+\ldots+v_k=d} (1+\sum_{i\in I} v_it_i)
\]
and extract from this its homogeneous component $\tau_{(d,r,k)}^ I$ of degree $(k+1)(r-k)$. Then 
\[
c_I=\tau^ I_{(d,r,k)}(-t_i)_{i\in I}=(-1)^{(r-k)(k+1)}\tau^ I_{(d,r,k)}(t_i)_{i\in I}.
\]
In conclusion we have
\begin{equation}\label{eq:Bott-1}
\deg (\Sigma(d,r,k))=\sum_{I\in \mathcal I_{k+1}} \frac {\tau^ I_{(d,r,k)}(t_i)_{i\in I}}{\prod_{i\in I}\prod_{j\not\in I}(t_i-t_j)}.
\end{equation}
As in \cite [p. 111] {H1}, we notice that the right hand side of this formula is, a priori, a rational function in the variables $t_0,\ldots,t_k$. As a matter of fact, it is a constant and a positive integer.

\section{Fano schemes of complete intersections } 
\label{sec:ci}
In this section we extend the considerations of Section \ref {sec:hyper} to complete intersections in projective space.
We consider the case in which a general complete intersection $X$ of type ${\bf d}:=(d_1,\ldots, d_m)$ in $\PP^ r$, where $\prod_{i=1}^m d_i>2$, does not contain any linear subspace of dimension $k$. Like in the case of hypersurfaces, the latter happens if and only if either $2k>r-m=\dim(X)$, or
\[
\gamma({\bf d}, r, k):= \sum_{j=1}^ m {{d_j+k}\choose k}-(k+1)(r-k)>0\,,
\]
see \cite{B, DM, L, Mi, Mo, Pr}.

Let $\Sigma({\bf d},r)$ be the parameter space for complete intersections of type ${\bf d}$ in $\PP^ r$. This is a tower of projective bundles over a projective space, hence a smooth variety. Consider the subvariety $\Sigma({\bf d},r,k)$ of $\Sigma({\bf d},r)$ parameterizing complete intersections which contain a linear subspace of dimension $k$. One has (\cite[Thm.\ 1.1]{BCFS}):

\begin {proposition}\label{prop:ci} Assume $\gamma({\bf d},r,k)>0$.
Then $\Sigma({\bf d},r,k)$ is a nonempty, irreducible and rational subvariety of codimension  $\gamma({\bf d},r,k)$ in $\Sigma({\bf d},r)$. The general point of $\Sigma({\bf d},r,k)$ corresponds to a complete intersection which contains a unique linear subspace of dimension $k$ and has singular locus of dimension $\max\{-1, 2k+m-1-r\}$ along its unique $k$-dimensional linear subspace (in particular, it is smooth provided $r\ge 2k+m$).
\end{proposition}

Next we would like to make sense of, and to compute, the \emph{degree} of $\Sigma({\bf d},r,k)$ inside $\Sigma({\bf d},r)$ when $\gamma({\bf d},r,k)>0$. To do this we consider the general complete intersection $X$ of type 
$(d_1,\ldots, d_{m-1})$, and the complete linear system 
\[
\Sigma(d_m,X)=|\mathcal O_X(d_m)|.
\]
We  assume that the Fano scheme $F_k(X)$ of linear subspaces of dimension $k$ contained in $X$ is non--empty. This implies that
\begin{equation}\label{eq:cix}
\dim(F_k(X))=(k+1)(r-k)-\sum_{j=1}^ {m-1} {{d_j+k}\choose k}={{d_m+k}\choose k}-\gamma({\bf d},r,k)\geqslant 0
\end{equation}
(see \cite {B,L, Pr}). Moreover, assume 
\begin{equation}\label{eq:ciy}
\dim(\Sigma(d_m,X))> \gamma({\bf d},r,k)\geqslant 0.
\end{equation}
Notice that  \eqref {eq:cix} and \eqref {eq:ciy} do hold if $\gamma({\bf d},r,k)$ is sufficiently small, e.g., if $\gamma({\bf d},r,k)=1$. 

Let now $\Sigma(d_m,X,k)$ be the set of points in $\Sigma(d_m,X)$ corresponding to complete intersections of type ${\bf d}=(d_1,\ldots,d_m)$ contained in $X$ and containing a subspace of dimension $k$. As an immediate consequence of Proposition \ref {prop:ci}, we have

\begin{corollary}\label{cor:ci} Assume $\gamma({\bf d},r,k)>0$ and \eqref {eq:cix} holds. Let $X$ be a general complete intersection of type $(d_1,\ldots, d_{m-1})$ verifying \eqref {eq:ciy}. Then
 $\Sigma(d_m,X,k)$ is irreducible of codimension  $\gamma({\bf d},r,k)$ in $\Sigma(d_m,X)$. The general point of $\Sigma(d_m,X,k)$ corresponds to a complete intersection  of type ${\bf d}=(d_1,\ldots,d_m)$ which contains a unique subspace of dimension $k$.
\end{corollary}

Next we would like to compute the degree of $\Sigma(d_m,X,k)$ inside the projective space $\Sigma(d_m,X)$.

Consider the vector bundle ${\rm Sym}^ {d_{m}} (\mathcal S^ *)$
on $\mathbb G(k,r)$ and set 
\[
\rho:={{d_m+k}\choose k}-\gamma({\bf d},r,k)=\dim(F_k(X)).
\]
Similarly as in Lemma \ref {lem:Chern}, one sees that
\begin{equation}\label{eq:deg}
\deg(\Sigma(d_m,X,k))= \int_{\mathbb G(k,r)} c_\rho({\rm Sym}^ {d_{m}} (\mathcal S^ *))\cdot [F_k(X)]\,,
\end{equation}
where $ [F_k(X)]$ stands for the dual class of  $F_k(X)$ in the Chow ring $A^*(\mathbb G(k,r))$.

\subsection{Schubert calculus}  
The  Chern class $c_\rho({\rm Sym}^ {d_{m}} (\mathcal S^ *))$ is the homogeneous component $\theta$ of degree $\rho$ of the polynomial 
\begin{equation*}\label{eq:pol}
\prod_{v_0+\ldots+v_k=d_m}(1+v_0x_0+\ldots+v_kx_k).
\end{equation*}
As usual, $\theta$ can be written as a polynomial in the elementary symmetric functions of the Chern roots $x_0,\ldots, x_k$, which can be  identified with the $\sigma_{(1^ i)}$s. Eventually, one has a formula of the form
\[
c_\rho({\rm Sym}^ {d_{m}}(\mathcal S^ *))=\sum_{j_1+2j_2+\ldots+(k+1)j_{k+1}=\rho} \phi_{d_m,r}(j_1,j_2,\ldots, j_{k+1}) \sigma_{(1)}^ {j_1}\sigma_{(1^ 2)}^ {j_2}\cdots 
\sigma_{(1^ {k+1})}^ {j_{k+1}}.
\]
In conclusion one has
\begin{equation}\label{eq:inu}
\begin{split}
&\deg(\Sigma(d_m,X,k))=\\
&=\int_{\mathbb G(k,r)} [F_k(X)]\cdot\sum_{j_1+2j_2+\ldots+(k+1)j_{k+1}=\rho} \phi_{d_m,r}(j_1,\ldots, j_{k+1}) \sigma_{(1)}^ {j_1}
\cdots 
\sigma_{(1^ {k+1})}^ {j_{k+1}}\,.
\end{split}
\end{equation}

\subsection {Debarre--Manivel's trick}\label {ssec:dmt} Formula \eqref {eq:inu} is rather unpractical, since both, the computation of the coefficients and of the intersection products appearing in it are rather complicated, in general.  A better result can be gotten using again Debarre--Manivel's idea as in \S \ref {ssec:DM}. Taking into account \eqref {eq:deg} 
one sees that  $\deg(\Sigma(d_m,X,k))$ equals the coefficient of the monomial $x_0^ rx_1^ {r-1}\cdots x_k^ {r-k}$ in the product of the following polynomials in $x_0,\ldots, x_k$:\begin{inparaenum}\\
\item [(i)] the product $Q_{k,{\bf d}}=\prod_{i=1}^{m-1} Q_{k,d_i}$ of the polynomials
\[
Q_{k,d_i}=\prod_{v_0+\ldots+v_k=d_i} (v_0x_0+\cdots+v_kx_k);
\]
\item [(ii)] the polynomial $\theta$;\\ 
\item [(iii)] the Vandermonde polynomial $V(x_0,\ldots, x_k)$.\\
\end{inparaenum}\\
Notice (\cite[14.7]{Ful}, \cite[3.5.5]{Man}) that $Q_{k,{\bf d}}$ in (i) corresponds to the class $[F_k(X)]$ of degree $(k+1)(r-k)-\rho$ 
in the Chow ring $A^*(\mathbb G(k,r))$,  
whereas 
$\theta$ in (ii) corresponds to the class of  $c_\rho({\rm Sym}^ {d_{m}} (\mathcal S^ *))$ of degree $\rho$.
In conclusion,
\[
\deg(\Sigma(d_m,X,k))=\psi_{(r,r-1,\ldots,r-k)}( Q\cdot \theta\cdot V).
\]

The Bott residue formula does not seem to be applicable in this case.

\section{Numerical invariants of Fano schemes}\label{sec:num-inv}

In this section we consider the complete intersections whose Fano schemes have positive expected dimension
\begin{equation}\label{eq:delta}
\delta({\bf d}, r, k):=-\gamma({\bf d}, r, k)=(k+1)(r-k)-\sum_{j=1}^ m {{d_j+k}\choose k}>0 \quad
\end{equation}
where ${\bf d}=(d_1,\ldots,d_m)$. We may and we will assume $d_i\ge 2$, $i=1,\ldots,m$. 
If also $r\ge 2k+m+1$ then, for a general complete intersection $X$ of type ${\bf d}$ in $\mathbb P^ r$, the Fano variety $F_k(X)$ of linear subspaces of dimension $k$ contained in $X$ is a smooth, irreducible variety of dimension $\delta({\bf d}, r, k)$ (see \cite {B, CV, DM, L, Mo, Pr}). We will  compute some numerical invariants of $F_k(X)$. If $\delta({\bf d}, r, k)=1$ then $F_k(X)$ is a smooth curve; its genus was computed in \cite {H2}. In the next section we treat the case where $F_k(X)$ is a surface, that is, $\delta({\bf d}, r, k)=2$; our aim is to compute the Chern numbers of this surface. Actually, we deduce formulas for $c_1(F_k(X))$ and $c_2(F_k(X))$ for the general case $\delta({\bf d}, r, k)>0$. To simplify the notation, we set in the sequel $F=F_k(X)$,  $\mathbb G=\mathbb G(k,r)$, $\delta=\delta({\bf d}, r, k)$, and we let  $\mathfrak h$ be the hyperplane section class of $\mathbb G$ in the Pl\"ucker embedding.

  Recall the following fact (cf.\ Proposition \ref{prop:DB-trick}).
  
  \begin{proposition}\label{prop:DM-formula} {\rm (\cite [Thm.\ 4.3]{DM})} In the notation and assumptions as before, one has
  \[
\deg(F)=\psi_{(r,r-1,\ldots,r-k)}(Q_{k,{\bf d}}\cdot e^{\delta}\cdot V)\quad\text{where}\quad e({\bf x}):=x_0+\cdots+x_k\,,
\]
  that is, the degree of the Fano scheme $F$ under the Pl\"ucker embedding equals the coefficient of the monomial $x_0^rx_1^{r-1}\cdots x_k^{r-k}$ of the product of $Q_{k,{\bf d}}\cdot e^{\delta}\cdot V$ where $V$ stands for the Vandermonde polynomial (see Subsection \ref {ssec:dmt} for the notation). 
  \end{proposition}
  
\begin{remark}\label{rem:Hiep} An alternative expression for $\deg(F)$  based on the Bott residue formula can be found in \cite[Thm.\ 1.1]{H1} and \cite [Thm.\ 2]{H2}; cf.\ also \cite[Ex.\ 14.7.13]{Ful} and \cite[Sect.\ 3.5]{Man}. 
\end{remark}

The next lemma is known in the case of the Fano scheme of lines on a general hypersurface, that is, for $k=m=1$, see \cite{AK}, \cite[Ex.\ V.4.7]{Kol}.
  
\begin{lemma}\label{lem:c1} In the notation and assumptions as before, one has
\begin{equation}\label{eq:c-1}
c_1(T_F)=\Big ( r+1-\sum_{i=1}^ m {{d_i+k} \choose {k+1}}\Big ) \mathfrak h_{|F}
\end{equation}
and
\begin{equation}\label{eq:K_F}
 K_F\sim   \mathcal  O_F\Big ( \sum_{i=1}^ m {{d_i+k} \choose {k+1}} -(r+1)\Big )\,
\end{equation}
where $\mathcal O_F(1)$ corresponds to the Pl\"ucker embedding. In particular, 
$F$ is a smooth Fano variety provided  $\sum_{i=1}^ m {{d_i+k} \choose {k+1}} \le r$.
\end{lemma}
\begin{proof}
From the exact sequence
\[
0\rightarrow T_F\rightarrow T_{\mathbb G|F}\rightarrow N_{F|\mathbb G}\rightarrow 0
\]
one obtains
\begin{equation*}\label{eq:c1}
c(T_{\mathbb G|F})=c(T_F)\cdot c(N_{F|\mathbb G})\,.
\end{equation*}
Expanding one gets
\begin{equation}\label{eq:c2}
c_1(T_F)=c_1(T_{\mathbb G|F})-c_1(N_{F|\mathbb G})
\end{equation}
and, for the further usage, 
\begin{equation}\label{eq:ccc2}
c_2(T_F)=c_2(T_{\mathbb G|F})-c_2(N_{F|\mathbb G})-c_1(T_{\mathbb G|F})\cdot c_1(N_{F|\mathbb G})+c_1(N_{F|\mathbb G})^ 2.
\end{equation}
Notice (\cite[Thm.\ 3.5]{EH}) that $T_\mathbb G= \mathcal S^ *\otimes \mathcal Q$, where, as usual, $\mathcal S\to \mathbb G$ is the tautological vector bundle of rank $k+1$ and $\mathcal Q\to\mathbb G$ is the tautological quotient bundle. Furthermore (\cite [Lemma 3]{H2}), $F$ is the zero scheme of a section of the vector bundle $\oplus_{i=1}^ m {\rm Sym}^ {d_i} (\mathcal S^ *)$ on $\mathbb G$. It follows that 
\begin{equation}\label{eq:x0} N_{F|\mathbb G}\simeq \oplus_{i=1}^ m {\rm Sym}^ {d_i} (\mathcal S^ *)_{|F}\,. \end{equation}

By \cite [Lemma 2]{H2} one has
\begin{equation}\label{eq:xx}
c_1(T_\mathbb G)=(r+1) \mathfrak h.
\end{equation}
Taking into account \eqref{eq:x0},
\cite [Lemma 1]{H2} (see also Lemma \ref{lem:comp} below), and the fact that $c_1(\mathcal S^ *)=\mathfrak h$ (see \cite[Sect.\ 4.1]{EH}), one gets
\begin{equation}\label{eq:n}
c_1(N_{F|\mathbb G})
=\sum_{i=1}^ m c_1({\rm Sym}^ {d_i} (\mathcal S^ *)_{|F})=
\Big(\sum_{i=1}^ m {{d_i+k} \choose {k+1}}\Big) \mathfrak h_{|F}.
\end{equation}
Plugging \eqref{eq:xx} and \eqref {eq:n} in \eqref {eq:c2} we find \eqref{eq:c-1} and then \eqref{eq:K_F}. 
\end{proof}
\begin{corollary} One has
\begin{equation}\label{eq:k2}
K^ \delta_F=\Big ( \sum_{i=1}^ m {{d_i+k} \choose {k+1}} -(r+1)\Big )^ \delta\deg(F),
\end{equation}
where $\deg(F)$ is computed in Proposition \ref{prop:DM-formula}.
\end{corollary}

Next we proceed to compute $c_2(T_F)$. Recalling \eqref {eq:ccc2}, we need to compute $c_2(N_{F|\mathbb G})$ and $c_2(T_\mathbb G)$. This requires some preliminaries. First of all, we need the following auxiliary combinatorial formula.

\begin{lemma}\label{lem:combinatorial-formulas} For any integers  $n,m,k$ where $n\ge m\ge 1$ and $k\ge 0$ one has
\begin{equation*}\label{eq:combin-ident}
\sum_{i=1}^n {{i-1}\choose {m-1}} {{n-i+k}\choose {k}} = {{n+k}\choose {m+k}}\,.\end{equation*} 
\end{lemma}

\begin{proof}\footnote{The authors are grateful to Roland Basher for communicating this beautiful, elementary argument.} 
The choice of $m+k$ integers $i_1,\ldots,i_{m+k}$ among 
$\{1,\ldots,n+k\}$, where $1\le i_1<\ldots<i_m<\ldots<i_{m+k}\le n+k$, 
can be done in two steps. At the first step one fixes the choice of $i_m=i$, where, clearly, $i\in\{1,\ldots,n\}$. It remains to choose $i_1,\ldots,i_{m-1}$ among $\{1,\ldots,i-1\}$ and $i_{m+1},\ldots,i_{m+k}$ among $\{i+1,\ldots,n+k\}$. 
\end{proof}

\begin{lemma}\label{lem:comp} Let $E$ be a vector bundle of rank $k+1$. Then
\begin{equation}\label{eq:cc2}
c_2({\rm Sym}^n(E))= \alpha c_1(E)^ 2+\beta c_2(E)\quad\text{and}\quad  
c_1({\rm Sym}^n(E))= \gamma c_1(E)
\end{equation}
where \footnote{See also \cite[Lemma 1]{H2} for $\gamma$.}
\begin{equation}\label{eq:ab}
\begin{split}  
\alpha &=
\frac{1}{2}{{n+k}\choose {k+1}}^2-\frac{1}{2}{{n+k}\choose {k+1}}-{{n+k}\choose {k+2}}
\,,\\
\beta &={{n+k+1}\choose {k+2}},\quad\text{and}\quad \gamma ={{n+k}\choose {k+1}}\,.
\end{split}
\end{equation}
\end{lemma} 
\begin{proof} We use the splitting principle. Write $E$ as a formal direct sum of line bundles $E=L_0\oplus\ldots\oplus L_k$, with $c_1(L_i)=x_i$, for $0\leqslant i\leqslant k$. From the equality 
\[c(E)=(1+x_0)\cdots (1+x_k)\] one deduces
\begin{equation}\label{eq:c1-c2} 
c_1(E)=x_0+\cdots+x_k\quad \text{and}\quad c_2(E)=\sum_{0\leqslant i<j\leqslant k}x_ix_j.
\end{equation} 
Since
\[
{\rm Sym}^n(E)=\sum_{v_0+\cdots+v_k=n}L_0^ {v_0}\cdots L_k^ {v_k}
\]
one has
\[
c({\rm Sym}^n(E))=\prod_{v_0+\cdots+v_k=n}(1+v_0x_0+\cdots+v_kx_k)=\prod_{|{\bf v}|=n} (1+\langle {\bf v}, {\bf x}\rangle)\,,
\]
where ${\bf x}=(x_0,\ldots,x_k)$, ${\bf v}=(v_0,\cdots,v_k)$, and $|{\bf v}|=v_0+\cdots+v_k$. 
Therefore,
\begin{equation}\label{eq:17}
c_1({\rm Sym}^n(E))=\sum_{|{\bf v}|=n} \langle {\bf v}, {\bf x}\rangle
\end{equation}
and
\begin{equation}\label{eq:18}
c_2({\rm Sym}^n(E))=\frac{1}{2}\sum_{|{\bf v}|=|{\bf w}|=n, {\bf v}\neq {\bf w}}  \langle {\bf v}, {\bf x}\rangle \langle {\bf w}, {\bf x}\rangle\,.
\end{equation}
The right hand sides of \eqref{eq:17} and \eqref{eq:18} are symmetric homogeneous polynomials  in $x_0,\ldots, x_k$ of degree 1 and 2, respectively. Using \eqref{eq:c1-c2} one deduces
\[\label{eq-c1} c_1({\rm Sym}^n(E))=\sum_{|{\bf v}|=n} \langle {\bf v}, {\bf x}\rangle
=\gamma(x_0+\cdots+x_k)=\gamma c_1(E)\,\]
and
\[\begin{aligned}\label{eq:aabb}
c_2({\rm Sym}^n(E)) &=\frac{1}{2}\sum_{|{\bf v}|=|{\bf w}|=n, {\bf v}\neq {\bf w}}  \langle {\bf v}, {\bf x}\rangle \langle {\bf w}, {\bf x}\rangle\\
&=\alpha(x_0+\cdots+x_k)^ 2+\beta \sum_{0\leqslant i<j\leqslant k}x_ix_j= \alpha c_1(E)^ 2+\beta c_2(E),
\end{aligned}\]
cf.\  \eqref{eq:cc2}. In order to compute $\alpha, \beta$ and $\gamma$, we 
let in these relations $x_0=1, x_1=\ldots=x_k=0$, so that the coefficient of $\beta$ vanishes and the coefficients of $\alpha$ and $\gamma$ become 1. Similarly, for $x_0=x_1=1, x_2=\ldots=x_k=0$ the coefficient of $\beta$ in the decomposition of $c_2({\rm Sym}^n(E))$ is 1 and the coefficient of $\alpha$ is 4.
So, one gets
\begin{equation}\label{eq:reduction}
\begin{split}
 \alpha=\frac{1}{2}\sum_{|{\bf v}|=|{\bf w}|=n, {\bf v}\neq {\bf w}} v_0w_0,\quad &\beta+4\alpha=\frac{1}{2}\sum_{|{\bf v}|=|{\bf w}|=n, {\bf v}\neq {\bf w}} (v_0+v_1)(w_0+w_1),\\
 &\gamma=\sum_{|{\bf v}|=n}v_0\,.
 \end{split}\end{equation}

For $k=1$, \eqref{eq:reduction} yields
\[
\begin{split}
\alpha &=\frac{1}{2}\Big(\sum_{i,j=1}^n ij-\sum_{i=1}^n i^2\Big)=\frac{1}{2}\Big(\left(\sum_{i=1}^n i\right)^2-\sum_{i=1}^n i^2\Big)\\
&=\frac{1}{2}\left(\frac{n^2(n+1)^2}{4}-\frac{n(n+1)(2n+1)}{6}\right)=\frac{(3n+2)}{4}{{n+1}\choose{3}}\,,
\end{split}
\]
and
\[
\begin{split}
\beta+4\alpha &= \frac{1}{2}\sum_{v_0+v_1=w_0+w_1=n} n^2-\frac{1}{2}\sum_{v_0+v_1=n} n^2 \\ &=   \frac{1}{2}n^2(n+1)^2- \frac{1}{2}n^2(n+1)=\frac{1}{2}n^3(n+1)\,.
\end{split}
\]
Plugging in the value of $\alpha$ gives
\[\label{eq:k=1}
\beta=\frac{1}{2}n^3(n+1)-4\alpha=\frac{1}{2}n^3(n+1)-(3n+2){{n+1}\choose{3}}={{n+2}\choose{3}}\,.\]

Similarly, if $k=2$ one has
\[\alpha=\frac{1}{2}\sum_{i,j=1}^n ij(n-i+1)(n-j+1)-\sum_{i=1}^n i^2(n-i+1)=\frac{5}{3}(n+1){{n+3}\choose{5}}\,\]
and 
\[\beta=\frac{1}{2}\left(\sum_{i,j=1}^n i(i+1)j(j+1)-\sum_{i=1}^n i^2(i+1)\right)-4\alpha={{n+3}\choose{4}}\,.\]

In the general case, 
applying Lemma \ref{lem:combinatorial-formulas} with a suitable choice of parameters
we find 
\[\gamma=  \sum_{i=1}^n  i{{n-i+k-1}\choose {k-1}} = {{n+k}\choose {k+1}}\] and 
\[\begin{split} \alpha & =  \frac{1}{2}
\sum_{i,j=1}^n   ij{{n-i+k-1}\choose {k-1}}  {{n-j+k-1}\choose {k-1}}- \frac{1}{2}\sum_{i=1}^n  i^2{{n-i+k-1}\choose {k-1}} \\
&= \frac{1}{2} \Big(\sum_{i=1}^n i {{n-i+k-1}\choose {k-1}}  \Big)^2-\sum_{i=1}^n  {{i+1}\choose{2}}{{n-i+k-1}\choose {k-1}}+\\
&+\frac{1}{2} \sum_{i=1}^n  i{{n-i+k-1}\choose {k-1}}\\
&=\frac{1}{2}{{n+k}\choose {k+1}}^2-{{n+k+1}\choose {k+2}}+\frac{1}{2}{{n+k}\choose {k+1}}=\frac{1}{2}{{n+k}\choose {k+1}}^2-\frac{1}{2}{{n+k}\choose {k+1}}-\\
&-{{n+k}\choose {k+2}}\,,
\end{split}\] 
where at the last step one uses   the standard identity 
\begin{equation}\label{eq:st-id}
{{N+1}\choose {k+1}}={{N}\choose {k+1}}+{{N}\choose {k}}\,.
\end{equation}
 Applying Lemma \ref{lem:combinatorial-formulas} and the identity 
\[i^2(i+1)=2{{i+1}\choose 2}+6{{i+1}\choose 3}\,,\]
for $k\ge 3$ we find:
\[\begin{split} \beta+4\alpha & =  \frac{1}{2}\sum_{i,j=1}^n   i(i+1)j(j+1){{n-i+k-2}\choose {k-2}}  {{n-j+k-2}\choose {k-2}}\\ & -\frac{1}{2}\sum_{i=1}^n   i^2(i+1){{n-i+k-2}\choose {k-2}}
 =  2\Big(\sum_{i=1}^n   {{i+1}\choose {2}}{{n-i+k-2}\choose {k-2}} \Big)^2\\&-\sum_{i=1}^n  {{i+1}\choose {2}}{{n-i+k-2}\choose {k-2}}-3\sum_{i=1}^n  {{i+1}\choose {3}}{{n-i+k-2}\choose {k-2}}\\ 
&=2{{n+k}\choose {k+1}}^2-{{n+k}\choose {k+1}}-3{{n+k}\choose {k+2}}\,.
\end{split}\]
Using the formula for $\alpha$ and \eqref{eq:st-id}
we deduce
\[
\beta ={{n+k}\choose {k+1}} + {{n+k}\choose {k+2}}={{n+k+1}\choose {k+2}}\,.
\] 
\end{proof}

\begin{remark}\label{rem:simplify} The proof shows that
for $k=1,2$, \eqref{eq:ab} can be simplified as follows:
\begin{equation}\label{eq:simplify}
(\alpha,\beta) =\begin{cases} 
\Big(\frac{3 n+2}{4}{{n + 1}\choose{3}},\,{{n+2}\choose{3}}\Big),&\quad \quad k=1\,,\\
\Big(\frac{5(n+1)}{3}{{n + 3}\choose{5}},\,{{n+3}\choose{4}}\Big),&\quad \quad k=2\,.\\
\end{cases}
\end{equation}
One can readily check that the expressions for $\alpha$ in these formulas agree with the one in \eqref{eq:ab}.
\end{remark}

\begin{lemma}\label{eq:q} One has
\begin{equation*}\label{eq:qs}
\begin{split}
c_1(\mathcal Q)&=c_1(\mathcal S^ *)=\mathfrak h\qquad\mbox{and}\\
c_2(\mathcal Q)&=c_1(\mathcal S^*)^ 2-c_2(\mathcal S^ *)=\mathfrak h^ 2-c_2(\mathcal S^ *).
\end{split}
\end{equation*}
\end{lemma}
\begin{proof} One has $c(\mathcal Q)\cdot c(\mathcal S)=1$. By expanding and taking into account  that 
$c_i(\mathcal S)=(-1)^ic_i(\mathcal S^ *)$ for all positive integers $i$ and $c_1(\mathcal S^ *)=\mathfrak h$ (\cite[Sect.\ 4.1]{EH}), the assertion follows. 
\end{proof}

\begin{lemma}\label{lem:qs2} One has
\begin{equation}\label{eq:TG}
c_2(T_\mathbb G)=\Big({{r+1}\choose 2}+k\Big)\mathfrak h^ 2 + \left(r-2k-1\right)c_2(\mathcal S^ *)\,.
\end{equation}
\end{lemma}
\begin{proof} We use again the splitting principle. Write
\[
\mathcal S^ *=L_0\oplus\cdots \oplus L_k, \quad \mathcal Q= M_1\oplus\cdots \oplus M_{r-k}
\]
with $c_1(L_i)=x_i, c_1(M_j)=y_j$, for $0\leqslant i\leqslant k$ and $1\leqslant j\leqslant r-k$. Since  $T_\mathbb G=\mathcal Q \otimes\mathcal S^ * $, see \cite[Thm.\ 3.5]{EH}, one obtains
\[
c(T_\mathbb G)=c(\mathcal Q\otimes \mathcal S^ *)= \prod_{i=0}^k\prod _{j=1}^{r-k}(1+x_i+y_j),
\]
whence
\[
c_2(T_\mathbb G)=\frac{1}{2}\sum_{\substack{\lambda,\mu=0,\ldots,k\\ \sigma,\rho=1,\ldots, r-k\\ (\lambda,\sigma)\neq (\mu,\rho)}} (x_\lambda+y_\sigma)(x_\mu+y_\rho) \,.
\]
By expanding, we see that in $c_2(T_\mathbb G)$ appear the following summands:\\ \begin{inparaenum}
\item [$\bullet$] $\xi=\sum_{i=0}^ k x_i^ 2$ and $\eta= \sum_{i=1}^ {r-k} y_j^ 2$, the former appearing ${r-k}\choose 2$ times, the latter ${k+1}\choose 2$ times;\\
\item [$\bullet$] $c_2(\mathcal S^ *)=\sum_{0\leqslant i<j\leqslant k} x_ix_j$, $c_2(\mathcal  Q)=\sum_{1\leqslant i<j\leqslant r-k} y_iy_j$, the former appearing $(r-k)^ 2$ times, the latter $(k+1)^ 2$ times;\\
\item [$\bullet$] $c_1(\mathcal  Q)c_1(\mathcal S^ *)=\sum_{i=0}^ k\sum_{j=1}^ {r-k} x_iy_j$ appearing $(k+1)(r-k)-1$ times.
\end{inparaenum}

Using Lemma \ref{eq:q} one obtains
\[
\xi=\sum_{i=0}^ k x_i^ 2=(x_0+\ldots+x_k)^ 2-2\sum_{0\leqslant i<j\leqslant k} x_ix_j=c_1(\mathcal S^*)^ 2-2c_2(\mathcal S^ *)=\mathfrak h^ 2-2c_2(\mathcal S^ *),
\]
and similarly 
\[
\eta=c_1(\mathcal Q)^ 2-2c_2(\mathcal Q)=\mathfrak h^ 2-2(\mathfrak h^ 2-c_2(\mathcal S^ *))=2c_2(\mathcal S^ *)-\mathfrak h^ 2.
\]
Collecting these formulas and taking into account Lemma \eqref{eq:q} one arrives at:
\[\begin{split}
c_2(T_\mathbb G)&={{r-k}\choose 2}\xi+{{k+1}\choose 2}\eta+(r-k)^ 2c_2(\mathcal S^ *)+(k+1)^ 2c_2(\mathcal  Q)\\
&\quad+\left((k+1)(r-k)-1\right)c_1(\mathcal  Q)c_1(\mathcal S^ *)\\
&=\left[{{r-k}\choose 2}-{{k+1}\choose 2}\right]\left(\mathfrak h^ 2-2c_2(\mathcal S^ *)\right)+(r-k)^ 2c_2(\mathcal S^ *)\\&\quad +(k+1)^ 2\left(\mathfrak h^ 2-c_2(\mathcal S^ *)\right)+((k+1)(r-k)-1)\mathfrak h^ 2\\ &=\left[{{r-k}\choose 2}-{{k+1}\choose 2}+(k+1)^ 2+(k+1)(r-k)-1\right]\mathfrak h^ 2\\ &\quad+\left[2{{k+1}\choose 2}-2{{r-k}\choose 2}+(r-k)^ 2-(k+1)^ 2\right]c_2(\mathcal S^ *)\\ &=\Big({{r+1}\choose 2}+k\Big)\mathfrak h^ 2 + \big(r-2k-1\big)c_2(\mathcal S^ *)\,.
\end{split}\]
\end{proof}

Now we can deduce the following formulas.

\begin{lemma}\label{lem:qs3} Let $\alpha_i$, $\beta_i$, and $\gamma_i$ be obtained from $\alpha$, $\beta$, and $\gamma$ in \eqref{eq:ab} by replacing 
$n$ by $d_i$, $i=1,\ldots,m$. Then one has
\begin{equation}\label{eq:c23} 
c_2(F)=c_2(T_F)=\Big (A\mathfrak h^ 2+B c_2(\mathcal S^ *)\Big)\cdot [F]
\end{equation}
where  $[F]$ is the class  of $F$ in the Chow ring $A^*(\mathbb G)$, and \footnote{The sum $\sum_{1\le i < j\le m} \gamma_i\gamma_j$ disappears if $m=1$.}
\begin{equation}\label{eq:A}
\begin{split}
A &= {{r+1}\choose 2}+k-\sum_{i=0}^ m \alpha_i-\sum_{1\le i < j\le m} \gamma_i\gamma_j\\
&-(r+1)\cdot \sum_{i=1}^ m {{d_i+k}\choose {k+1}} + \Big( \sum_{i=1}^ m{{d_i+k}\choose {k+1}}\Big)^ 2\,,
\end{split}
\end{equation}
and 
\begin{equation}\label{eq:B}
B=r-2k-1-\sum_{i=1}^ m\beta_i\,.
\end{equation}
\end{lemma}

\begin{proof}
Using 
\eqref {eq:xx} and  \eqref {eq:n}
we deduce
\begin{equation}\label{eq: prod-c1}
c_1(T_{\mathbb G})|_F\cdot c_1(N_{F|\mathbb G})=\Big((r+1)\sum_{i=1}^m {{d_i+k}\choose {k+1}}\Big)\mathfrak h^ 2\cdot [F]\,\end{equation}
and
\begin{equation}\label{eq:c1-square} c_1(N_{F|\mathbb G})^2=\Big(\sum_{i=1}^m {{d_i+k}\choose {k+1}}\Big)^2\mathfrak h^ 2\cdot [F]\,.\end{equation}
Furthermore, the Whitney formula and Lemma \ref{lem:comp}  yield
\begin{equation*}\label{eq:c2-NF}
\begin{split} 
c_2(N_{F|\mathbb G}) &=\sum_{i=1}^m c_2\big(\Sym^{d_i}(\mathcal S^ *)|_F\big)+\sum_{1\le i<j\le m} c_1 \big(\Sym^{d_i}(\mathcal S^ *)|_F\big)\cdot c_1 \big(\Sym^{d_j}(\mathcal S^ *)|_F\big)\\ 
&= \Big(\sum_{i=1}^m \left( \alpha_i c_1(\mathcal S^ *)^2 +\beta_i c_2(\mathcal S^ *)\right)+ \sum_{1\le i<j\le m} \gamma_i\gamma_j c_1(\mathcal S^ *)^2 \Big)\cdot [F]\\
 &=  \Big(\sum_{i=1}^m \alpha_i + \sum_{1\le i<j\le m} \gamma_i\gamma_j\Big)\mathfrak h^ 2\cdot [F]
+ \Big(\sum_{i=1}^m\beta_i \Big)  c_2(\mathcal S^ *)\cdot [F]\,.
\end{split} 
\end{equation*}
Plugging this in \eqref{eq:ccc2} together with the values of the Chern classes from \eqref{eq:TG}, 
\eqref{eq: prod-c1}, and \eqref{eq:c1-square}
gives \eqref{eq:c23}, \eqref{eq:A}, and \eqref{eq:B}. 
\end{proof}

\begin{remark}\label{rem:F}
The cycle $F$ on $\mathbb G$  is the reduced zero scheme of a section of the vector bundle $\mathcal{E}_F:=\oplus_{i=1}^ m {\rm Sym}^ {d_i} (\mathcal S^ *)$ on $\mathbb G$
of rank 
\[{\rm rk}(\mathcal{E}_F)={{{\bf d}+k}\choose{k} }:= \sum_{i=1}^m {{d_i+k}\choose{k}}\,.\] 
The Poincar\'e dual $[F]\in A^{{{{\bf d}+k}\choose{k} }}(\mathbb G)$ of the
class of $F$ in $A_\delta(\mathbb G)$ is the top Chern class $c_{{{{\bf d}+k}\choose{k} }}(\mathcal{E}_F)$. 
The latter  can be expressed in terms of the Chern roots as
\begin{equation*}\label{eq:F}
[F]=Q_{k,{\bf d}}(x_0,\ldots, x_k)=\prod_{i=1}^m Q_{k,d_i}(x_0,\ldots, x_k)\in A^{ {{\bf d}+k}\choose{k} } (\mathbb G)\,,
\end{equation*}
see Section \ref {ssec:dmt}.
\end{remark}

\section{The case of Fano surfaces}\label{sec:Fano-surf} Let us turn to the case where the Fano scheme $F=F_k(X)$ of a general complete intersection $X\subset\PP^r$ of type $\bf d$ is an irreducible surface, that is, $\delta=2$ and $r\ge 2k+m$. Let us make the following observations.
  
In the surface case, $\int_{F} c_2(F)=e(F)$ is the Euler--Poincar\'e characteristic of $F$. By Lemma \ref {lem:qs3} one can compute $e(F)$ once one knows $\int_{\mathbb G} \mathfrak h^ 2\cdot [F]$ and $\int_{\mathbb G} c_2(\mathcal S^ *)\cdot [F]$. As for the former, one can use the Debarre-Manivel formula for the degree
$\deg(F)=\int_{\mathbb G}\mathfrak h^ 2\cdot [F]$, see Proposition \ref{prop:DM-formula}; cf.\ also Remark \ref{rem:Hiep}. 

As for the latter, recall that $c_2(\mathcal S^ *)=\sigma_{(1^ 2)}$ is the class of the Schubert cycle of the $\PP^ k$s in $\PP^r$ intersecting a fixed $\PP^ {r-k-1}$ in a line. Computing $\int_{\mathbb G} c_2(\mathcal S^ *)\cdot [F]$ geometrically is difficult. However, one can compute it using  Debarre--Manivel's trick. Indeed,  arguing as in the proof of \cite [Thm. 4.3]{DM}, cf.\ Subsection \ref{ssec:DM}, one can see that $\int_{\mathbb G} c_2(\mathcal S^ *)\cdot [F]$ equals the coefficient of $x_0^rx_1^{r-1}\cdots x_k^{r-k}$ in the product of the three factors:
\begin{itemize}\item
$Q_{k,{\bf d}}=\prod_{i=1}^ m Q_{k,d_i}$, see Subsection \ref {ssec:dmt};
\item $c_2(\mathcal S^ *)=\sum_{0\leqslant i<j\leqslant k}x_ix_j$; 
\item the Vandermonde polynomial $V(x_0,\ldots,x_k)=\prod_{i<j} (x_i-x_j)$.
\end{itemize}
Notice that for $\delta=2$ one has 
\[\deg\left(Q_{k,{\bf d}}\cdot\sum_{0\leqslant i<j\leqslant k}x_ix_j\right)={{{\bf d}+k}\choose{k}}+2=(k+1)(r-k)=\dim(\mathbb G)\,. \]

Putting together \eqref {eq:c23}, \eqref {eq:A} and \eqref {eq:B} one finds a formula for the Euler characteristic $e(F)=\int_{F} c_2(F)$. Then, using \eqref {eq:k2} and the Noether formula
\[
\chi(\mathcal O_F)=\frac 1{12}\Big (K^ 2_F+e(F)\Big)=\frac 1{12}\Big (c_1(F)^ 2+c_2(F)\Big)
\] 
one can compute the holomorphic Euler characteristic $\chi(\mathcal O_F)$, the arithmetic genus $p_a(F)=\chi(\mathcal O_F)-1$, and the signature $\tau(F)=4\chi(\mathcal O_F)-e(F)$. 

\begin{example}\label{ex:cubic-3fold}  Let us apply these recipes to the well known case of the Fano surface $F=F_1(X)$ of lines  on the general cubic threefold in $\PP^4$. Letting 
$r=4, k=m=1, d=3$ one gets $\delta=2$ and
\[Q_{1,(3)}(x_0,x_1)=9x_0x_1(2x_0+x_1)(x_0+2x_1),\quad V(x_0,x_1)=x_0-x_1\,.\] Therefore,
\[\deg(F)=\int_{\mathbb G(1,4)} \mathfrak{h}^2\cdot[F] =\int_{\mathbb G(1,4)} c_1(\mathcal S^ *)^2\cdot[F] =\psi_{4,3}(Q_{1,(3)}\cdot (x_0+x_1)^2\cdot V)=45\,\] and
\[\int_{\mathbb G(1,4)} c_2(\mathcal S^ *)\cdot [F]=\psi_{4,3}(Q_{1,(3)}\cdot x_0x_1\cdot V)=27\,.\]
Applying \eqref{eq:simplify} and \eqref{eq:c23} one obtains
\[\alpha=11,\quad \beta=10,\quad A=6,\quad\text{and}\quad B=-9\,.\]
Using the Noether formula and \eqref {eq:k2} one arrives at the classical values (see \cite{AK}, \cite{Li})
\[e(F)=c_2(F)=6\deg(F)-9\int_{\mathbb G(1,4)} c_2(\mathcal S^ *)\cdot [F]=6\cdot 45 - 9 \cdot 27=27\,\] and
\[c_1(F)^2=K_F^2=\Big({{4}\choose{2}}-5\Big)^2 \deg(F)=45,\,\,\, \chi(\mathcal{O}_F)=\frac{1}{12}(45+27)=6\,.\]
\end{example}

\begin{example}\label{ex:hypersurface 2r-5}  More generally, one can consider the Fano surface $F=F_1(X)$ of lines on
a general hypersurface $X$ of degree $d=2r-5$ in $\PP^r$, $r\ge 4$. Plugging in \eqref{eq:c23}-\eqref{eq:B} the values of $\alpha$ and $\beta$ from  \eqref{eq:simplify}  one obtains
\[
\begin{split}A&={{2r-4}\choose{2}}^2+{{r+1}\choose{2}}+1-\frac{6r-13}{4}{{2r-4}\choose{3}}-(r+1){{2r-4}\choose{2}}\,,\\
&=\frac{1}{3} (6 r^4 - 56 r^3 + 177 r^2 - 211 r + 78)\,,\\ B&=r-3-{{2r-3}\choose{3}}\,.
\end{split}
\]
Furthermore,
\[e(F)=c_2(F)=A\deg(F)+B\int_{\mathbb G(1,r)} c_2(\mathcal S^ *)\cdot [F],\quad c_1^2(F)=\Big({{2r-4}\choose{2}}-(r+1)\Big)^2\,,\]
and
\[\chi(\mathcal O_F)=\frac{1}{12}\left(c_1^2(F)+c_2(F)\right)\,,\]
where 
\[\deg(F)=\psi_{r,r-1}\left(Q_{1,(d)}\cdot  (x_0+x_1)^2(x_0-x_1)\right)\]
and 
\[\int_{\mathbb G(1,r)} c_2(\mathcal S^ *)\cdot [F]=\psi_{r,r-1}\left(Q_{1,(d)}\cdot  x_0x_1(x_0-x_1)\right)\]
with 
\[Q_{1,(d)}=\prod_{v_0+v_1=d} (v_0x_0+v_1x_1)\,.\]

Consider, for instance, the Fano scheme $F$  of lines on a general quintic fourfold in $\PP^5$. One has
\[r=5,\,\,d=5,\,\,k=m=1,\,\,\delta=2\,.\]
One gets
\[\alpha=85,\quad\beta=35,
\quad A=66,\quad B=-33\,, \]
and further (cf.\ \cite[Table 1]{DM} and \cite{vdW})
\[\deg(F)=\psi_{5,4}\left(Q_{1,(5)}\cdot  (x_0+x_1)^2(x_0-x_1)\right)=25\cdot 245=6125\]
and
\[c_2(\mathcal S^ *)\cdot [F]=\psi_{5,4}\left(Q_{1,(5)}\cdot  x_0x_1(x_0-x_1)\right)=25\cdot 115= 2875\,.\]
Hence
\[e(F)=c_2(F)=25\cdot 33\cdot 375=309375,\quad c_1^2(F)=25\cdot 81\cdot 245=496125\,.\]
Finally,
\[\chi(\mathcal O_F)=\frac{1}{12}\left(c_1^2(F)+c_2(F)\right)=25\cdot 15\cdot 179=67125\,.\]

\end{example}

\begin{example}\label{ex:two-quadrics}  Consider further the Fano surface $F=F_1(X)$ of lines  on the intersection $X$ of two general quadrics in $\PP^5$. 
We have \[r=5,\,\,\, m=2,\,\,\, {\bf d}=(2,2), \,\,\,k=1,\,\,\, \delta=2\,,\] 
\[Q_{1,(2,2)}(x_0,x_1)=16x_0^2x_1^2(x_0+x_1)^2,\quad\text{and}\quad V(x_0,x_1)=x_0-x_1\,.\]
Hence 
\[\deg(F) =\psi_{5,4}(Q_{1,(2,2)}\cdot (x_0+x_1)^2\cdot V)=32\,\] and
\[\int_{\mathbb G(1,5)} c_2(\mathcal S^ *)\cdot [F]=\psi_{5,4}(Q_{1,(2,2)}\cdot x_0x_1\cdot V)=16\,.\]
Furthermore,
\[\alpha_1=\alpha_2=2,\quad \beta_1=\beta_2=4,\quad \gamma_1=\gamma_2=3,\quad A=3,\quad\text{and}\quad B=-6\,.\]
Therefore,
\[e(F)=3\deg(F)-6\int_{\mathbb G(1,5)} c_2(\mathcal S^ *)\cdot [F]=3\cdot 32- 6\cdot 16=0\,,\] 
\[c_1(F)^2=\Big(2{{3}\choose{2}}-6\Big)^2 \deg(F)=0,\quad\text{and so,}\quad\chi(\mathcal{O}_F)=0 \,.\]
In fact, $F$ is an abelian surface (\cite{Re}).  
\end{example}

\begin{example}\label{ex:cubic-fivefold} Let now $F=F_2(X)$ be  the Fano scheme of planes on a general cubic fivefold $X$ in $\PP^6$. Thus, one has 
\[r=6, \,\,\,m=1,\,\,\, d=3, \,\,\,k=2,\quad\text{and}\quad \delta=2\,.\]
 Letting
 \small
\[Q_{2,(3)}=27x_0x_1x_2(2x_0+x_1)(2x_0+x_2)(x_0+2x_1)(x_0+2x_2)(2x_1+x_2)(x_1+2x_2)(x_0+x_1+x_2)\,\]
\normalsize
and
\[ V(x_0,x_1,x_2)=(x_0-x_1)(x_0-x_2)(x_1-x_2)\, .\]
The Wolfram Alpha gives (cf.\ \cite[Table 2]{DM})
\[\deg(F)=\psi_{6,5,4}(Q_{2,(3)}\cdot(x_0+x_1+x_2)^2\cdot V)=27\cdot 105=2835\,\]
and
\[\int_{\mathbb G(2,6)} c_2(\mathcal S^ *)\cdot [F]=\psi_{6,5,4}(Q_{2,(3)}\cdot (x_0x_1+x_0x_2+x_1x_2)\cdot V)=27\cdot 63=1701\,.\]
Standard calculations yield
\[\alpha=40,\quad \beta=15,\quad \gamma=10, \quad A=13, \quad\text{and}\quad B=-14\,.\]
So, one obtains
\[e(F)=13\deg(F)-14 c_2(\mathcal S^ *)\cdot [F]=13041\,,\]
\[K_F^2=\Big ( {{5} \choose {3}} -7 \Big )^ 2\deg(F)=9\deg(F)=9\cdot 27\cdot 105=25515\,,\]
 and
\[\chi(\mathcal O_F)=\frac{13041+25515}{12}=3213\,.\]
\end{example}

\section{Irregular Fano schemes}\label{sec:irreg-Fano}
In this section we study the cases in which the Fano scheme $F$ of a general complete intersection is irregular, that is, $q(F)=h^1(\O_F)>0$. As follows from the next proposition, for the Fano surfaces $F$ this occurs only if $F$ is one of the surfaces in Examples \ref{ex:cubic-3fold} (or, which is the same, in \ref{ex:hypersurface 2r-5} for $r=4$), \ref{ex:two-quadrics}, and \ref{ex:cubic-fivefold}. In all these cases one has $r=2k+m+1$.

\begin{thm}\label{thm:irregular} Let $X$ be a general complete intersection of type ${\bf d}=(d_1,\ldots,d_m)$ in $\PP^r$. Suppose that the Fano scheme $F=F_k(X)$ of $k$-planes in $X$, $k\ge 1$, is irreducible of dimension $\delta\ge 2$. Then $F$ is irregular 
if and only if one of the following holds:
\begin{itemize}
\item[{\rm (i)}] $F$ is the variety of lines on a general cubic threefold in $\PP^4$ {\rm ($\dim(F)=2$)};
\item[{\rm (ii)}] $F$ is the variety of planes on a general cubic fivefold in $\PP^6$  {\rm ($\dim(F)=2$)};
\item[{\rm (iii)}] $F$ is the variety of $k$-planes on the intersection of two general quadrics in $\PP^{2k+3}$, $k\in\NN$  {\rm ($\dim(F)=k+1$)}. 
\end{itemize}
\end{thm}

\begin{proof} By our assumption, $\delta \ge2$. 
By \cite [Thm.\ 3.4]{DM} one has $q(F)=0$ if $r\geqslant 2k+m+2$. 
By \cite [Thm.\ 2.1]{DM}, $F$ being nonempty implies $r\geqslant 2k+m$. 
Therefore, $q(F)>0$ leaves just two possibilities: 
\[ r= 2k+m\quad\text{and}\quad r= 2k+m+1\,.\]
We claim that the first possibility is not realized. Indeed, let $r= 2k+m$. We may assume that $d_i\ge 2$ for all $i=1,\ldots, m$. From \eqref{eq:delta} one deduces:
\begin{equation}\label{eq:new-delta} (k+1)(k+m)=(k+1)(r-k)=\delta+\sum_{i=1}^m {{d_i+k}\choose{k}}\ge 2+m{{k+2}\choose{2}}\,.\end{equation}
 This implies the inequality
\begin{equation}\label{eq:delta-ineq} 4\le 
k(k+1)(2-m)\,,\end{equation}
and so, $m=1$,
that is, $X$ is a hypersurface in $\PP^{2k+1}$. Letting $d=d_1$, \eqref{eq:new-delta} reads
\[\delta=(k+1)^2-{{d+k}\choose{k}}\ge 2\,.\] This inequality holds only when $d=2$, that is, $X$ is a smooth quadric of dimension $2k$. However, in the latter case $F=F_k(X)$ consists of two components (\cite[Lemma 1.1]{Do}), contrary to our assumption. This proves our claim. 

In the case $r= 2k+m+1$, \eqref{eq:new-delta} and \eqref{eq:delta-ineq} must be replaced, respectively, by
\begin{equation*}\label{eq:new-delta1} (k+1)(k+m+1)=(k+1)(r-k)=\delta+\sum_{i=1}^m {{d_i+k}\choose{k}}\ge \delta+m{{k+2}\choose{2}}\,\end{equation*}
and 
\begin{equation}\label{eq:delta-ineq1} 4\le 2\delta\le (k+1)[2(k+m+1)-m(k+2)]=(k+1)[(2-m)k+2]\,.\end{equation} 
 It follows from \eqref{eq:delta-ineq1} that either $m=1$  and $r=2k+2$, or $m=2$ and $r=2k+3$.

In the hypersurface case (i.e., $m=1$) one has
\[2\le\delta=(k+1)(k+2)-{{d+k}\choose{k}}\,.\] 
This holds only if either $d=2$, or $d\ge 3$ and $k\in\{1,2\}$. 

If $d=2$, that is, $X$ is a smooth quadric  in $\PP^{2k+2}$, then $\delta={{k+2}\choose{2}}$, cf.\ \cite[Lemma 1.1]{Do}. However, by  \cite[Lemma 1.2]{Do}, in this case $F$ is unirational, hence $q(F)=0$, contrary to our assumption. 

The possibility $d\ge 3$ realizes just in the following two cases:
\begin{itemize}\item[{\rm (i)}] $(d,r,k)=(3,4,1)$, that is, $F$ is the Fano surface of lines on a smooth cubic threefold in $\PP^4$; 
\item[{\rm (ii)}] $(d,r,k)=(3,6,2)$, that is, $F$ is the Fano surface of planes on a smooth cubic fivefold in $\PP^6$.
\end{itemize}
If further $m=2$ then 
$r=2k+3$ and \[2\le\delta=(k+1)(k+3)-{{d_1+k}\choose{k}}-{{d_2+k}\choose{k}}.\] 
This inequality holds only for ${\bf d}=(2,2)$, that is, only in the case where
\begin{itemize}\item[{\rm (iii)}] $F=F_k(X)$ is the Fano scheme of $k$-planes in a smooth intersection of two quadrics in $\PP^{2k+3}$. 
\end{itemize}
Notice that $F$ as in (iii) is smooth, irreducible, of dimension $\delta=k+1$, cf.\ \cite[Ch.\ 4]{Re} and Remarks \ref{rem:even-dim} below.

It remains to check that $q(F)>0$  in  (i)-(iii) indeed. 
 
The Fano surface $F=F_1(X)$ of lines on a smooth cubic threefold $X\subset \PP^4$ in (i) was studied by Fano (\cite{F}) who found, in particular, that $q(F)=5$. From Example \ref {ex:cubic-3fold}, we deduce that $p_g(F)=10$ (cf.\ also \cite[Thm.\ 4]{Bea},  \cite{BSD}, \cite{CG}, \cite{Gh}, 
\cite{Li}, \cite[Sect. 4.3]{Re}, \cite{Ru}, \cite{T1}, \cite{T2}). There is an isomorphism ${\rm Alb}(F) \simeq J(X)$ where $J(X)$ is the intermediate Jacobian (see \cite{CG}). The latter holds as well for $F=F_2(X)$ where $X\subset \PP^6$ is a smooth cubic fivefold as in (ii), see \cite{Co}. Thus, $q(F)>0$ in (i) and (ii). 

By a  theorem of M.~Reid \cite[Thm.\ 4.8]{Re} (see also \cite[Thm.\ 2]{DR}, \cite{T4}), the Fano scheme $F=F_k(X)$ of $k$-planes on a smooth intersection $X$ of two quadrics in $\PP^{2k+3}$ as in (iii) is isomorphic to the Jacobian $J(C)$ of a hyperelliptic curve $C$ of genus $g(C)=k+1$ (of an elliptic curve if $k=0$). Hence, one has $q(F)=\dim(F)=k+1>0$ for $k\ge 0$. Notice that there are isomorphisms $F\simeq J(C)\simeq J(X)$ where $J(X)$ is the intermediate Jacobian, see \cite{Do}.
\end{proof}

\begin{remarks}\label{rem:even-dim} 1. The complete intersections in (i)-(iii) are Fano varieties. The ones in (i) are the Fano threefolds of index $2$ with a very ample generator of the Picard group. The complete intersections Fano threefolds of index $1$ with a very ample anticanonical divisor are the varieties $V_3^{2g-2}\subset \PP^{g+1}$ of genera  $g=3, 4, 5$, that is, the smooth quartics $V_3^4$ in $\PP^3$ ($g=3$),  the smooth intersections $V_3^6$ of a quadric and a cubic  in $\PP^5$ ($g=4$), and the smooth  intersections $V_3^8$ of three quadrics in $\PP^6$ ($g=5$), see \cite[Ch. IV, Prop.\ 1.4]{Is}. 
The Fano scheme of lines $F=F_1$ on a general such Fano threefold 
$V_3^{2g-2}$ is a smooth curve of a positive genus $g(F)>0$. In fact, $g(F)=801$ for $g=3$, $g(F)=271$ for $g=4$, and $g(F)=129$ for $g=5$, see \cite{Mar} and \cite[Examples 1-3]{H2}. 
For these $X=V_3^{2g-2}$, the Abel-Jacobi map 
$J(F) \to J(X)$ to the intermediate Jacobian is an epimorphism, and $J(X)$ coincides with the Prym variety of $X$, see \cite{Is} and \cite[Lect.\ 4, Sect.\ 1, Ex.\ 1 and Sect.\ 3]{T3}.

2. Notice that the complete intersections whose Fano schemes of lines are curves of positive genera  are not exhausted by the above Fano threefolds $V_3^{2g-2}$. The same holds, for instance, for a general hypersurface of degree $2r-4$ in $\PP^r$, $r\ge 4$, and for  general complete intersections of types ${\bf d}=(r-3,r-2)$ and ${\bf d}=(r-4,r-4)$  in $\PP^r$ for $r\ge 5$ and $r\ge 6$, respectively, see \cite[Examples 1-3]{H2}, etc. One can find in \cite{H2} a  formula for the genus of the curve $F$. 

3. Let $X$ be a smooth intersection of two quadrics in $\PP^{2k+2}$. Then the Fano scheme $F_k(X)$ is reduced and finite of cardinality $2^{2k+2}$ (\cite[Ch.\ 2]{Re}), whereas $F_{k-1}(X)$ is a rational Fano variety of dimension $2k$ and index $1$, whose Picard number is $\rho = 2k + 4$, see \cite{AC}, \cite{Ca}, and the references therein.  
\end{remarks}

As for the Picard numbers of the Fano schemes of complete intersections, 
one has the following result (cf.\ also \cite{DM}). 

\begin{theorem} {\rm (\cite[Thm.\ 03]{ZJ})} Let $X$ be a very general complete intersection in $\PP^r$. Assume
$\delta({\bf d}, r,  k) \ge 2$. 
Then $\rho(F_k(X)) = 1$ except in the following cases:
\begin{itemize}\item
$X$ is a quadric in $\PP^{2k+1}$, $k \ge 1$. Then $F_k(X)$ consists of two  isomorphic smooth disjoint components, and the Picard number of each component is $1$;
\item $X$ is a quadric in $\PP^{2k+3}$, $k \ge 1$. Then $\rho(F_k(X))=2$; 
\item $X$ is a complete intersection of two quadrics in $P^{2k+4}$, $k \ge 1$. Then $\rho(F_k(X))=2k + 6$.
\end{itemize}
\end{theorem}

The assumption ``very general'' of this theorem cannot be replaced by ``general''; one can find corresponding examples in \cite{ZJ}.

\section{Hypersurfaces containing conics}\label{sec:conics}
Recall (see \cite[Thm.\ 1.1]{Fur}) that for the general hypersurface $X$ of degree $d$ in $\PP^r$, the variety $R_2(X)$ of smooth conics in $X$ is smooth\footnote{and connected provided $\mu_2\ge 1$ and $X$ is not a smooth cubic surfaces in $\PP^3$.}  of the expected dimension $\mu(d,r) = 3r-2d-2$ provided $\mu(d,r) \ge 0$, and is empty otherwise. In this section we concentrate on the latter case.

\subsection{The codimension count and uniqueness}\label{ss:conic-codim}

Set \[\epsilon(d,r)=2d+2-3r\,.\] Consider the subvariety $\Sigma_c(d,r)$ of $\Sigma(d,r)$ whose points correspond to hypersurfaces containing plane conics. By abuse of language, in the sequel we say ``conic'' meaning ``plane conic''; thus, a pair of skew lines does not fit in our terminology. 
A conic is smooth if it is reduced and irreducible.

\begin{theorem}\label{lem:tbpp} Assume $d\ge 2$, $r\ge 3$, and $\epsilon(d,r)\ge 0$. Then the following hold.
\begin{itemize}
\item[{\rm (a)}]  $\Sigma_c(d,r)$ is irreducible of codimension  $\epsilon(d,r)$ in $\Sigma(d,r)$. 
\item[{\rm (b)}]  If $\epsilon(d,r)> 0$ and $(d,r)\neq (4, 3)$ then the hypersurface corresponding to the general point of $\Sigma_c(d,r)$ contains a unique conic, and this conic is smooth.
In the case $(d,r)=(4, 3)$ it contains exactly two distinct conics, and these conics  are smooth and coplanar.
\end{itemize}
\end{theorem}
\begin{proof} (a) Let $\mathcal H_{c,r}$ be the component of the Hilbert scheme whose points parameterize conics in $\PP^ r$. There is an obvious morphism
\[
\pi: \mathcal H_{c,r}\rightarrow \mathbb G(2,r)
\]
sending a conic $\Gamma$ to the plane $\Pi=\langle \Gamma \rangle$. The fibers of $\pi$ are projective spaces of dimension 5, hence $\mathcal H_{c,r}$ is a $\PP^ 5$-bundle over $\mathbb G(2,r)$. Therefore $ \mathcal H_{c,r}$ is a smooth, irreducible projective variety of dimension $3r-1$. 

Consider the incidence relation
\[I=\{(\Gamma, X)\in \mathcal H_{c,r} \times \Sigma(d,r)\,|\,\Gamma\subset X\}\,\]
and the natural projections 
\[p\colon I\to \mathcal H_{c,r}\quad\mbox{and}\quad q\colon I\to\Sigma(d,r)\,.\] 
It is easily seen that $q(I)=\Sigma_c(d,r)$ and that, for any $\Gamma\in\mathcal H_{c,r}$, 
$p^{-1}(\Gamma)$ is a linear subspace of $\{\Gamma\}\times \Sigma(d,r)$ of codimension $2d+1$. Indeed, $\Gamma$ being a complete intersection, it is projectively normal, hence the restriction map
\[
H^0(\PP^r, \O_{\PP^r}(d))\to H^0(\Gamma, \O_{\Gamma}(d))\cong \CC^{2d+1}
\]
is surjective. It follows
that $I$ and $\Sigma_c(d,r)$ are irreducible proper schemes. Moreover, one has
\begin{equation*}\label{eq:dim-I}
\begin{split}
& \dim (I)=\dim (p^{-1}(\Gamma))+\dim (\mathcal H_{c,r})=\\
&={{d+r}\choose {d}}-1-(2d+2-3r)= \dim(\Sigma(d,r))-\epsilon(d,r)\,.
 \end{split}\end{equation*}
Letting $\kappa(d,r)$ be the dimension of the general fiber of $q\colon I\to\Sigma_c(d,r)$, one obtains
$$\dim(\Sigma_c(d,r))=\dim (I)-\kappa(d,r)= \dim(\Sigma(d,r))-\epsilon(d,r)-\kappa(d,r)\,,$$ and therefore
 $$\codim(\Sigma_c(d,r), \Sigma(d,r))=\epsilon(d,r)+\kappa(d,r)\,.$$ 
 
 Next we prove that $\kappa(d,r)=0$, which will accomplish the proof of part (a). To do this, we imitate the argument in \cite [p. 29]{B}. 
 
 First of all, consider again the surjective morphism $q: I\to \Sigma_c(d,r)$. Since $I$ is irreducible, the general element $(\Gamma,X)\in I$ maps to the general element $X\in \Sigma_c(d,r)$. Since $(\Gamma,X)\in I$ is general and $p\colon I\to \mathcal H_{c,r}$ is surjective, then $\Gamma$ is smooth. Hence the general  $X\in \Sigma_c(d,r)$ contains some smooth conic $\Gamma$. Moreover, the general fibre of $q$ could be reducible, but, by Stein factorization, all components of it are of the same dimension and  exchanged by monodromy. This implies that the general element $(\Gamma,X)$ of any component of  $q^{-1}(X)$ with $X\in \Sigma_c(d,r)$ general, is such that $\Gamma$ is smooth (cf. Claim \ref{cl:sm} below for an alternative argument). 
 
 By choosing appropriate coordinates, we may assume that  if  $(\Gamma,X)$ is the general element of a component of  $q^{-1}(X)$ with $X\in \Sigma_c(d,r)$ general, then $\Gamma$ has equations
 \[
 x_0x_1-x_2^2=x_3=\cdots=x_r=0
 \]
 and $X$ has equation  $F=0$ with
 \[
 F=A( x_0x_1-x_2^2)+B_3x_3+\cdots +B_rx_r+R
 \]
 where 
\[
A=\sum_{v_0+v_1+v_2=d-2}\alpha_{\bf v}x_0^{v_0}x_1^{v_1}x_2^{v_2},  B_i=\sum_{w_0+w_1+w_2=d-1}\beta_{\bf w}x_0^{w_0}x_1^{w_1}x_2^{w_2}, \quad \text{for}\quad i=3,\ldots r
\]
and $R\in I_\Gamma^2$. By Bertini's theorem we may assume that $X$ is smooth. We have the normal bundles sequence
\[
0\to N_{\Gamma|X}\to N_{\Gamma|\PP^r}\cong \O_{\Gamma}(2)\oplus \O_{\Gamma}(1)^{\oplus r-3} \to {N_{X|\PP^r}}_{|\Gamma}\cong \O_\Gamma(d)\to 0
\]
We want to show that $h^0(N_{\Gamma|X})=0$, which implies that $\kappa(d,r)=0$, as desired.
In order to prove this, we will prove that the map
\[
\varphi: H^0(N_{\Gamma|\PP^r}) \to H^0({N_{X|\PP^r}}_{|\Gamma})
\]
is injective. Notice that $h^0(N_{\Gamma|\PP^r})=3r-1$ and $h^0({N_{X|\PP^r}}_{|\Gamma})=2d+1$, and so, the assumption $\epsilon(d,r)\geq 0$ reads $h^0(N_{\Gamma|\PP^r})\leq h^0({N_{X|\PP^r}}_{|\Gamma})$. 

We can interpret a section in $H^0(N_{\Gamma|\PP^r})$ as the datum of $(f,f_3,\ldots, f_r)$, where $f\in H^0(\O_{\Gamma}(2)))$ is a homogeneous polynomial
\[
f=\sum_{0\leq i\leq j\leq 2} b_{ij}x_ix_j
\]
 taken modulo  $x_0x_1-x_2^2$, and $f_i\in H^0(\O_{\Gamma}(1)))$ is a linear form
\[
f_i=a_{i0}x_0+a_{i1}x_1+a_{i2}x_2, \quad \text{for}\quad i=3,\ldots, r.
\]
Notice that the parameters on which $(f,f_3,\ldots, f_r)$ depends are indeed $3r-1$, namely the $3(r-2)$  coefficients $a_{ij}$s plus the 5 coefficients $b_{ij}$s. The map $\varphi$ sends 
 $(f,f_3,\ldots, f_r)$ to the restriction of $Af+B_3f_3+\cdots +B_rf_r$ to $\Gamma$. By identifying $\Gamma$ with $\PP^1$ via the map sending $t\in \PP^1$ to
\begin{equation}\label{eq:raf}
x_0=1, x_1=t^2, x_2=t, x_3=\cdots =x_r=0
\end{equation}
the restriction of $Af+B_3f_3+\cdots +B_rf_r$ to $\Gamma$ identifies (after the substitution \eqref {eq:raf}) with a polynomial $P(t)$ of degree $2d$ in $t$. Let us order the $3(r-2)$  coefficients $a_{ij}$s and the 5 coefficients $b_{ij}$s in such a way that the $b_{ij}$s come before the  $a_{ij}$s, and inside each group they are ordered lexicographically. Then we can  consider the matrix $\Phi$ of the map $\varphi$, which is of type $(2d+1)\times (3r-1)$. Indeed, each one of the $2d+1$ coefficients of the polynomial $P(t)$ of degree $2d$ is in turn a polynomial of the $b_{ij}$ and $a_{ij}$. Given  $b_{ij}$ or $a_{ij}$, its coefficients in those polynomials form the corresponding column of $\Phi$.

Notice that the latter coefficients (that is, the entries of $\Phi$) are linear functions of the $\alpha_{\bf v}$s and the $\beta_{\bf w}$s. Moreover, in each row and in each column of $\Phi$ a given $\alpha_{\bf v}$ and a given $\beta_{\bf w}$ appear at most once. 
 
 The map $\varphi$ is injective if and only if $\Phi$ has  rank $3r-1$ for sufficiently general values of  the $\alpha_{\bf v}$s and the $\beta_{\bf w}$s. We will in fact consider the $\alpha_{\bf v}$s and the $\beta_{\bf w}$s as indeterminates and prove that there is a maximal minor of $\Phi$, e.g., the one $\Phi'$ determined by the first $3r-1$ rows, which is a non--zero polynomial in these variables. This will finish our proof. 
 
Consider, for example, the order of the $\alpha_{\bf v}$s and the $\beta_{\bf w}$s in which the former come before the latter and in each group they are ordered lexicographically.  Let us order the monomials appearing in the expression of $\Phi'$  according to the following rule: 
the monomial $m_1$ comes before the monomial $m_2$ if for the smallest variable appearing in $m_1$ and in  $m_2$ with different exponents, the exponent in $m_1$ is larger than the exponent in $m_2$.  The greatest monomial in this ordering will have coefficient $\pm 1$ in $\Phi'$, since in each row, the choice of the $\alpha_{\bf v}$s and the $\beta_{\bf w}$s entering in it is prescribed. This proves that $\Phi'\neq 0$. 

(b) We have to show that, if $\epsilon(d,r)> 0$ and, except for $(d,r)=(4, 3)$, the hypersurface $X$ corresponding to the general point of $\Sigma_c(d,r)$ contains a unique conic. To do this we use counts of parameters,  which show that the codimension in  $\Sigma(d,r)$ of the locus of 
hypersurfaces $X$ containing at least two distinct conics is strictly larger than $\epsilon(d,r)$. The proof is a bit tedious, since it requires to consider a number of different possibilities, namely that two conics on $X$ do not intersect, or they intersect in one, two or in four points (counting with multiplicity). We will not treat in detail all the cases, but only the former and the latter, leaving some easy details in the remaining two cases to the reader, which could profit from similarity with the dimension count we made at the beginning of this proof.

We start with the following two claims.
 
\begin{claim}\label{cl:2l} 
The subset $ \Sigma_{2l}(d,r)$ of all the $X\in \Sigma_c(d,r)$ such that $X$ contains a double line is a proper subvariety of $\Sigma_c(d,r)$. 
\end{claim}

\begin{proof}[Proof of Claim \ref {cl:2l}] 
Consider the closed subset $\mathcal H_{2l,r}\subset \mathcal H_{c,r}$ whose general point corresponds to a double line in $\PP^r$. There is a natural $\PP^2$-fibration $\mathcal H_{2l,r}\to\mathbb{G}(2,r)$. Hence one has $\dim(\mathcal H_{2l,r})=3r-4$. Consider further the incidence relation 
\[I_{2l}=\{(\Gamma, X)\in \mathcal H_{2l,r}\times\Sigma(d,r)\,\vert\,\Gamma\in X\}\]
with projections $p_{2l}, q_{2l}$ to the first and the second factors, respectively. The general fiber $F_{2l}$ of $p_{2l}$ is a linear subspace of $\Sigma(d,r)$ of codimension 
$2d+1$. It follows that $I_{2l}$ is an irreducible projective variety of dimension
\[\dim(I_{2l})=\dim(\Sigma(d,r))-(2d+3-3r)\,.\] Therefore, the image
$\Sigma_{2l}(d,r)=q_{2l}(I_{2l})$
is an irreducible proper subvariety of $\Sigma(d,r)$ of codimension at least 
\[2d+3-3r=\epsilon(r,d)+1=\codim(\Sigma_c(d,r),\Sigma(d,r))+1\,,\] see 
(a).
\end{proof}

We know by (a) that if $X\in  \Sigma_c(d,r)$ is general, then $X$ contains only finitely many conics (recall that  the general fiber of $q:I\to \Sigma_c(d,r)$ has dimension $\kappa(d,r)=0$). Our next claim is the following. 

\begin{claim}\label{cl:sm} 
The conics contained in the general $X\in \Sigma_c(d,r)$ are all smooth. 
\end{claim}

\begin{proof}[Proof of Claim \ref {cl:sm}] 
The incidence variety $I$ being irreducible, the monodromy group of the generically finite morphism $q:I\to \Sigma_c(d,r)$ acts transitively on the general fiber $q^{-1}(X)$. Its action on $I$ lifts to the universal family of conics over $I$. The latter action by homeomorphisms of the general fiber (which consists of a finite number of conics) preserves the Euler characteristic. We know already that the general $X$ contains a smooth conic. Due to Claim \ref{cl:2l}, $X$ does not carry any double line. Since the Euler characteristic (equal $3$) of the union of two crossing lines is different from the one of a smooth conic,  all the conics in $X$ are smooth. 
\end{proof}

Suppose now the general $X\in  \Sigma_c(d,r)$ contains more than one conic, and assume first it contains two conics which do not intersect. We will see this leads to a contradiction.

Let $\mathcal H_{cc,r}$ be the component of the Hilbert scheme whose general point corresponds to a pair of conics in $\PP^ r$ which do not meet. It is easy to see that  $ \mathcal H_{cc,r}$ is an irreducible projective variety of dimension $6r-2$. 

Consider the incidence relation
$$I=\{(\Gamma, X)\in \mathcal H_{cc,r} \times \Sigma(d,r)\,|\,\Gamma\subset X\}\,$$
and the natural projections $$p\colon I\to \mathcal H_{cc,r}\quad\mbox{and}\quad q\colon I\to\Sigma(d,r)\,.$$

\begin{claim}\label{cl:plu} For any $\Gamma\in\mathcal H_{cc,r}$ which corresponds to a pair of disjoint smooth conics, $p^{-1}(\Gamma)$ is a linear subspace of $\{\Gamma\}\times \Sigma(d,r)$ of codimension $4d+2$. 
\end{claim}

\begin{proof}[Proof of Claim \ref {cl:plu}] 
By our assumption, $r\ge 3$. Then the hypothesis $\epsilon(d,r)>0$ implies $d\geq 4$. So, we have to prove that  the restriction map
\[\rho: H^0(\PP^r, \O_{\PP^r}(d))\to H^0(\Gamma, \O_{\Gamma}(d))\cong \CC^{4d+2}\,,\]
where $\Gamma=\Gamma_1\cup\Gamma_2$ with $\Gamma_1, \Gamma_2$ disjoint smooth 
conics, is surjective as soon as $d\geq 4$. Actually we will prove it for $d\geq 3$. By projecting generically into $\PP^3$, it suffices to prove the assertion if $r=3$.  

The  restriction map 
\[H^0(\PP^3, \O_{\PP^r}(d))\to H^0(\Gamma_1, \O_{\Gamma_1}(d))\cong \CC^{2d+1}\]
is surjective for all $d\geq 1$ because $\Gamma_1$ is projectively normal. Hence the kernel of this map, i.e., $H^0(\PP^3, \mathcal I_{\Gamma_1|\PP^3}(d))$ has codimension $2d+1$ in $H^0(\PP^3, \O_{\PP^3}(d))$. Consider now the restriction map
\[
\rho': H^0(\PP^3, \mathcal I_{\Gamma_1|\PP^3}(d))\to H^0(\Gamma_2, \O_{\Gamma_2}(d))\cong \CC^{2d+1}.
\]
We will prove that this map is also surjective. This will imply that its kernel, i.e., 
$H^0(\PP^3, \mathcal I_{\Gamma|\PP^3}(d))$ has codimension $2d+1$ in  $H^0(\PP^3, \mathcal I_{\Gamma_1|\PP^3}(d))$, hence it has codimension $4d+2$ in $H^0(\PP^3, \O_{\PP^r}(d))$, which proves that  $\rho$ is surjective. 

Let $\Pi_i=\langle \Gamma_i\rangle$ be the plane spanned by $\Gamma_i$, for $i=1,2$. Consider the intersection scheme $\mathfrak D$ of $\Gamma_1$ with $\Pi_2$, so that $\mathfrak D$ is a zero dimensional scheme of length 2, and $\mathfrak D$ is not contained in $\Gamma_2$.  
To prove that $\rho'$ is surjective, notice that it is composed of the following two restriction maps
\[
\rho_1: H^0(\PP^3, \mathcal I_{\Gamma_1|\PP^3}(d))\to H^0(\Pi_2, \mathcal I_{\mathfrak D|\Pi_2}(d)),
\]
\[
\rho_2: H^0(\Pi_2, \mathcal I_{\mathfrak D|\Pi_2}(d))\to H^0(\Gamma_2, \O_{\Gamma_2}(d)).
\]
The map $\rho_1$ is surjective, because its cokernel is $H^1(\PP^3, \mathcal I_{\Gamma_1|\PP^3}(d-1))$ which is zero because $\Gamma_1$ is projectively normal. 
So, we are left to prove that the map $\rho_2$ is surjective. The kernel of $\rho_2$ is 
$H^0(\Pi_2, \mathcal I_{\mathfrak D|\Pi_2}(d-2))$, whose dimension is
\[
h^0(\Pi_2, \mathcal I_{\mathfrak D|\Pi_2}(d-2))={{d}\choose 2}-2
\]
as soon as $d\geq 3$. Similarly 
\[
h^0(\Pi_2, \mathcal I_{\mathfrak D|\Pi_2}(d))={{d+2}\choose 2}-2
\]
for any $d\geq 1$. Hence the dimension of the image of $\rho_2$ is
\[
h^0(\Pi_2, \mathcal I_{\mathfrak D|\Pi_2}(d))-h^0(\Pi_2, \mathcal I_{\mathfrak D|\Pi_2}(d-2))={{d+2}\choose 2}-{{d} \choose 2}=2d+1
\]
which proves that $\rho_2$ is surjective. \end{proof}

By Claim \ref {cl:plu}, $I$ has a unique component $I'$ which dominates $\mathcal H_{cc,r}$, and 
\[
\begin{split}
\dim(I')&= \dim (\mathcal H_{cc,r})+\dim (\Sigma(d,r))- (4d+2)=\\
&=\dim (\Sigma(d,r))-(4d-6r+4).
\end{split}\]
Since we are assuming $q(I')=\Sigma_c(d,r)$, we have 
\[
\begin{split}
&\dim (\Sigma(d,r))-(4d-6r+4)=\dim(I')\geq \\
&\geq \dim(q(I'))=\dim (\Sigma(d,r))-\epsilon (d,r)= \dim (\Sigma(d,r))-(2d-3r+2),
\end{split}
\]
whence $\epsilon (d,r)=2d+2-3r\leq 0$, contrary to the assumption  $\epsilon (d,r)> 0$. 

A similar argument works also in the cases where the general $X\in  \Sigma_c(d,r)$ contains two smooth conics which meet in one or two points, counting with multiplicity. The corresponding closed subset $\mathcal H^{(i)}_{cc,r}\subset \mathcal H_{cc,r}$ whose general point corresponds to a pair of smooth conics which meet in $i$ points, where $i=1,2$, is an irreducible proper scheme of dimension $5r$ for $i=1$ and $4r+2$ for $i=2$. 
Letting $I^{(i)}\subset \mathcal H^{(i)}_{cc,r}\times\Sigma_c(d,r)$
be the corresponding incidence relation and arguing as in the proof of Claim \ref{cl:plu} one can easily show that any fiber of the  projection $I'\to \mathcal H^{(i)}_{cc,r}$ over a point $\Gamma\in \mathcal H^{(i)}_{cc,r}$ representing a pair of conics with exactly $i$ places in common counting with multiplicity,  is a linear subspace of $\Sigma(d,r)$ of codimension $4d+1$ if $i=1$ and $4d$ if $i=2$, where $I'$ is the unique component of $I^{(i)}$ which dominates $\Sigma_c(d,r)$. Proceeding as before, this leads in both cases to the inequality $r\le 2$, which contradicts the assumption $r\ge 3$. 

Consider finally the remaining (extremal) case in which the general $X\in  \Sigma_c(d,r)$  contains two conics which are coplanar, i.e., they intersect (counting with multiplicity) at 4 points.

We denote by  $ \mathcal F(=\mathcal H^{(4)}_{cc,r})$ the subvariety of the Hilbert scheme whose general point corresponds to a pair of coplanar conics in $\PP^r$. It is easy to see that  $ \mathcal F$ is an irreducible projective variety of dimension $3r+4$. 

Consider the incidence relation
$$I=I^{(4)}=\{(\Gamma, X)\in \mathcal F \times \Sigma(d,r)\,|\,\Gamma\subset X\}\,$$
and the projections $$p\colon I\to \mathcal F\quad\mbox{and}\quad q\colon I\to\Sigma(d,r)\,.$$ For any $\Gamma\in\mathcal F$, $p^{-1}(\Gamma)$ 
is a linear subspace of $\{\Gamma\}\times \Sigma(d,r)$ of codimension $4d-2$. 
Indeed, since $\Gamma$ is a complete intersection, it is projectively normal. Hence
the restriction map
\[H^0(\PP^r, \O_{\PP^r}(d))\to H^0(\Gamma, \O_{\Gamma}(d))\]
is surjective for all  $d\geq 1$. On the other hand, $\Gamma$ is a curve of arithmetic genus $3$, and the dualizing sheaf of $\Gamma$ is $\O_{\Gamma}(1)$. Hence,  
$h^0(\Gamma, \O_{\Gamma}(d))=4d-3+1=4d-2$, as soon as $d\geq 2$. 

Thus, $I$ is irreducible, and 
\[
\begin{split}
\dim(I)&= \dim (\mathcal F)+\dim (\Sigma(d,r))- (4d-2)=\\
&=\dim (\Sigma(d,r))-(4d-3r-6).
\end{split}\]
Since we are assuming $q(I)=\Sigma_c(d,r)$, we have 
\[
\begin{split}
&\dim (\Sigma(d,r))-(4d-3r-6)=\dim(I)\geq \\
&\geq \dim(q(I))=\dim (\Sigma(d,r))-\epsilon (d,r)= \dim (\Sigma(d,r))-(2d+2-3r),
\end{split}
\]
whence $d\leq 4$. Since 
\[
0<\epsilon (d,r)=2d+2-3r\leq 10-3r
\]
we see that the only possibility is $d=4, r=3$. In this case a similar argument proves that the general $X\in \Sigma_c(4,3)$ contains exactly two coplanar, smooth conics. \end{proof}
   
\subsection{The degree count}\label{ss:conic-deg}
Next we compute the degree of $\Sigma_c(d,r)$ in $\Sigma(d,r)$, provided $\epsilon(d,r)>0$.

Let $f: \mathcal C\to \mathcal H_{c,r}$ be the universal family over $\mathcal H_{c,r}$, which is endowed with a map $g: \mathcal C\to \PP^ r$. 
We denote by $\mathcal E_d$ the vector bundle $f_*(g^ *(\mathcal O_{\PP^r}(d))$ over $ \mathcal H_{c,r}$. 
If $\Gamma$ is a conic, the fiber $\mathcal E_{d,\Gamma}$ of $\mathcal E_d$ at (the point corresponding to) $\Gamma$ is $H^ 0(\Gamma, \mathcal O_\Gamma(d))$.  We set $\mathcal E=\mathcal E _1$.
Notice that $\mathcal E_d$ is a vector bundle of rank $2d+1>3r-1=\dim (\mathcal H_{c,r})$; in particular, ${\rm rk}(\mathcal E)=3$. 

\begin{lemma}\label{lem:degg} If $\epsilon(d,r)>0$ and $(d,r)\neq (4,3)$, then
$$\deg(\Sigma_c(d,r))=\int_{\mathcal H_{c,r}} c_{3r-1}(\mathcal E_d)\,.$$
Moreover
$$\deg(\Sigma_c(4,3))=\frac 12\int_{\mathcal H_{c,3}} c_{8}(\mathcal E_4)\,.$$
 
\end{lemma}
\begin{proof}  
Any homogeneous form $F$ of degree $d$ in $r+1$ variables defines a section $\sigma_F$ of 
$\mathcal E_d$ such that $\sigma_F(\Gamma)=F|_\Gamma\in H^0(\Gamma,\O_\Gamma(d))$. Consider the effective divisor $X_F$ of degree $d$  on $\PP^r$ of zeros of $F$. The support of $X_F$ contains $\Gamma$ if and only if $\sigma_F(\Gamma)=0$. Counting the conics $\Gamma\in\mathcal H_{c,r}$ lying in Supp$(X_F)$ is the same as counting the zeros of $\sigma_F$ in $\mathcal H_{c,r}$ with their multiplicities.
 
 Let further $\rho=\dim(\mathcal H_{c,r})=3r-1$. By our assumption one has  
 $${\rm rk}\,(\mathcal E_d)-\rho=\epsilon(d,r)>0\,.$$
Choose a general linear subsystem $\mathcal{L}=\langle X_0,\ldots,X_\epsilon\rangle$ in $|\O_{\PP^r}(d)|$ of dimension $\epsilon=\epsilon(d,r)$, where $X_i=\{F_i=0\}$. 
By virtue of Theorem \ref{lem:tbpp},  $\mathcal{L}$ meets $\Sigma (d,r,k)\subset\Sigma(d,r)$ transversally in $\deg(\Sigma_c(d,r))$ simple points.

Consider now  the sections $\sigma_i:=\sigma_{F_i}$, $i=0,\ldots,\epsilon$, of $\mathcal E_d$ and assume $(d,r)\neq (4,3)$.  By Theorem \ref{lem:tbpp}, the intersection of $\mathcal L$ with $\Sigma(d,r,k)$ is exactly the scheme of points $\Gamma\in \mathcal H_{c,r}$ where there is a linear combination of $\sigma_0,\ldots, \sigma_\epsilon$ vanishing on $\Gamma$.  This is the zero dimensional scheme of points of $\mathcal H_{c,r}$ where the sections $\sigma_0,\ldots, \sigma_\epsilon$ are linearly dependent. This zero dimensional scheme represents the top Chern class $c_\rho(\mathcal E_d)$ (see \cite [Thm. 5.3]{EH}).
Its degree is the top Chern number $\int_{\mathcal H_{c,r}} c_{3r-1}(\mathcal E_d)$. 

The case $(d,r)=(4,3)$ is similar: one has to take into account again Theorem \ref{lem:tbpp}, which says that the general quartic surface in $\PP^3$ 
contains  exactly two smooth conics, and these conics are coplanar.
\end{proof}

To compute $c_{3r-1}(\mathcal E_d)$ we proceed as follows. 
For a positive integer $d$ consider ${\rm Sym}^ d(\mathcal E)$. 
Note that the universal family $\mathcal C$ over $\mathcal H_{c,r}$ is the zero set of a section $\xi$ of  ${\rm Sym}^ 2(\mathcal E)$. 
For any 
$d\geqslant 2$ one has the exact sequence 
\[
0\rightarrow {\rm Sym}^ {d-2}(\mathcal E)\stackrel{\cdot \xi}\rightarrow {\rm Sym}^ {d}(\mathcal E)\rightarrow \mathcal E_d\rightarrow 0\,.
\]
Hence
\begin{equation}\label{eq:ccc}
c(\mathcal E_d)= c({\rm Sym}^ {d}(\mathcal E))\cdot c({\rm Sym}^ {d-2}(\mathcal E))^ {-1}\,.
\end{equation}

To compute Chern classes, as usual, we use the splitting principle. We write formally 
\[
\mathcal E=L_1\oplus L_2 \oplus L_3,
\]
the $L_i$s being (virtual) line bundles. Consider the Chern roots $x_i=c_1(L_i)$ of $\mathcal E$. One has
\[
c(\mathcal E)=(1+x_1)(1+x_2)(1+x_3)\,,
\] 
that is, 
\[
c_1(\mathcal E)=x_1+x_2+x_3, \quad c_2(\mathcal E)=x_1x_2+x_1x_3+x_2x_3, \quad  \quad c_3(\mathcal E)=x_1x_2x_3.
\]
Furthermore, 
\[
{\rm Sym}^ d(\mathcal E)=\bigoplus_{v_1+v_2+v_3=d}L_1^{v_1}L_2^{v_2} L_3^{v_3}.
\]
Letting \[c_1(L_1^{v_1}L_2^{v_2} L_3^{v_3})=v_1x_1+v_2x_2+v_3x_3=\langle {\bf v}, {\bf x}\rangle\]
where ${\bf x}=(x_1,x_2,x_3)$ and ${\bf v}=(v_1,v_2,v_3)$ with $|{\bf v}|=v_1+v_2+v_3$
 one obtains
\begin{equation*}\label{eq:ch1}
c({\rm Sym}^d(\mathcal E))=\prod_{|{\bf v}|=d}(1+\langle {\bf v}, {\bf x}\rangle
\end{equation*}
and, by \eqref {eq:ccc}, 
\begin{equation}\label{eq:ddd}
c(\mathcal E_d)
=\left(\prod_{|{\bf v}|=d}(1+\langle {\bf v}, {\bf x}\rangle)\right)\\ \cdot  \left(\prod_{|{\bf v}|=d-2}(1+\langle {\bf v}, {\bf x}\rangle)\right)^ {-1}\,.
\end{equation}
Now, the top Chern class $c_{3r-1}(\mathcal E_d)$ is the homogeneous component $\eta(x_1,x_2,x_3)$ of degree $3r-1$ in the right hand side of \eqref {eq:ddd} written as a formal power series  in $x_1,x_2,x_3$. This is a symmetric form of degree $3r-1$  in $x_1,x_2,x_3$. It can be expressed via the elementary symmetric polynomials in  $x_1,x_2,x_3$, i.e., in terms of $c_1(\mathcal E), c_2(\mathcal E), c_3(\mathcal E)$.

In order to compute $c_{3r-1}(\mathcal E_d)$ effectively, we prefere to use the Bott residue formula. 
The standard diagonal action of  $T={(\mathbb C^ *)}^ {r+1}$ on $\PP^ r$, see  \ref {ssec:bott}, 
induces an action of $T$ on $\mathbb G(2,r)$  and on $\mathcal H_{c,r}$. 

\begin{lemma}\label{lem:fp} The action of $T$ on $\mathcal H_{c,r}$ has exactly $r(r^ 2-1)$ isolated fixed points.
\end{lemma}
\begin{proof} Let $\Gamma$ be a fixed point for the $T$-action on $\mathcal H_{c,r}$. Then $\Pi=\langle \Gamma\rangle$ is fixed under the action of $T$ on $\mathbb G(2,r)$. Hence $\Pi$ is one of the coordinate planes in $\PP^ r$, and these are 
${r+1}\choose 3$ in number. We let $x,y,z$ be the three coordinate axes in $\Pi$. Then the only conics on $\Pi$ fixed by the $T$-action are the singular conics $x+y,x+z,y+z, 2x,2y,2z$. Thus, we get in total $6{{r+1}\choose 3}=r(r^ 2-1)$ fixed points of $T$ in $\mathcal H_{c,r}$. 
\end{proof}

We denote by $\mathcal F$ 
the set of fixed points for the $T$-action on $\mathcal H_{c,r}$.
Bott's residue formula, applied in our setting, has the form
\[
\deg (\Sigma_c(d,r))=\int_{\mathcal H_{c,r}} c_{3r-1}(\mathcal E_d)=\sum_{\Gamma\in \mathcal F} \frac {c_\Gamma}{e_\Gamma}\,,
\]
where $\frac {c_\Gamma}{e_\Gamma}$ is the local contribution of a fixed point $\Gamma\in \mathcal F$. Recall that $c_\Gamma$ results from the local contribution of 
$c_{3r-1}(\mathcal E_d)$ at $\Gamma$, and  $e_\Gamma$ is determined by the torus action on the tangent space to $\mathcal H_{c,r}$ at the point corresponding to $\Gamma$, see \ref{ssec:bott}. 

To compute $e_\Gamma$ we have to compute the characters of the $T$-action on the tangent space 
\[
T_{\Gamma}(\mathcal H_{c,r}) \simeq H^ 0(\Gamma, N_{\Gamma|\PP^ r})\simeq H^ 0(\Gamma, \mathcal O_\Gamma(1))^ {\oplus (r-2)} \oplus H^ 0(\Gamma, \mathcal O_\Gamma(2))\simeq \mathcal E_\Gamma^{\oplus {(r-2)}}\oplus \mathcal E_{2,\Gamma}\,.
\]
Let $\Pi=\langle \Gamma\rangle$. Then $\Pi$ is a coordinate plane which corresponds to a subset $I=\{i,j,k\}\subset\{0,\ldots, r\}$ consisting of $3$ distinct elements. Let $ \mathcal I_3$ be the set of all the ${r+1}\choose{3}$ such subsets $I$. The characters of the $T$-action on $\mathcal E_\Gamma$ have weights $-t_\alpha$ with $\alpha\in I$. Let $I^{(2)}$ be the symmetric square of $I$; it consists of $6$ unordered pairs $\{\alpha,\beta\}$, $\alpha, \beta\in I$. The 
characters of the $T$-action on $\mathcal E_{2,\Gamma}$ have weights $t_\alpha + t_\beta$ with $\{\alpha, \beta\}\in I^{(2)}\setminus\{a, b\}$, where $x_ax_b=0$ is the equation of $\Gamma$ in $\Pi$.  Then 
\[
e_\Gamma=(-1)^{3(r-2)} (t_it_jt_k)^{r-2}\prod_{\{\alpha, \beta\}\in I^{(2)}\setminus\{a, b\}}(t_\alpha+ t_\beta)\,.
\]
As for $c_\Gamma$, with the same notation as above we have
\[
c_\Gamma=\eta(-t_i,-t_j,-t_k)=(-1)^ {3r-1} \eta(t_i,t_j,t_k) \quad \text {where} \quad I=\{i,j,k\}.
\]
In conclusion we find the formula
\begin{equation}\label{eq:deg-conics}
\deg (\Sigma_c(d,r))= -\sum_{I=\{i,j,k\}\in \mathcal I_3}\sum_{\{a,b\}\in I^{(2)}}\frac {\eta(t_i,t_j,t_k)} {(t_it_jt_k)^{r-2}\prod_{\{\alpha, \beta\}\in I^{(2)}\setminus \{a, b\}} (t_\alpha+ t_\beta)}\,.
\end{equation}
Again, the right hand side of this formula is, a priori, a rational function in the variables $t_0,\ldots,t_r$. In fact, this is a constant and a positive integer. 
Letting $t_i=1$ for all $i=0,\ldots,r$ 
we arrive at the following conclusion.

\begin{thm} Assuming that $\epsilon(d,r)=2d+2-3r>0$ and $(d,r)\neq (4,3)$ one has
\begin{equation}\label{eq:deg-conics1} \deg (\Sigma_c(d,r))= -\frac{5}{32}{{r+1}\choose{3}}\eta(1,1,1)\,,\end{equation}
where $\eta$ is the homogeneous form of degree $3r-1$ in the formal power series 
decomposition of the right hand side of \eqref {eq:ddd}. 
\end{thm}

\begin{remark}\label{rem:r=3}
In the case of the surfaces in $\PP^3$, one can find in \cite[Prop.\ 7.1]{MaRoXaVa}  a formula for the degree of $\Sigma_c(d,3)$ expressed as a polynomial in $d$ for $d\ge 5$. This formula was deduced by applying Bott's residue formula. After dividing by 2, this formula gives also the correct value $\deg(\Sigma_c(4,3))=2508$. 
\end{remark}


%
%


\begin{thebibliography}{}

\bibitem {AK} A.~B.~Altman, S.~L.~Kleiman, \emph{Foundations of the theory of Fano schemes}, Compos.\ Math.\ {\bf  34} (1977), 3--47.

\bibitem {AC} C.~ Araujo, C.~Casagrande,  \emph{On the Fano variety of linear spaces contained in two odd-dimensional quadrics}, Geometry and Topology {\bf 21} (2017), 3009--3045.

\bibitem {BVdV} W.~Barth, A.~Van de Ven, \emph{Fano varieties of lines on hypersurfaces}, Arch.\ Math.\ (Basel), {\bf  31} (1978), 96--104.

\bibitem {BCFS} 
F.~Bastianelli, C.~Ciliberto, F.~Flamini, and P.~Supino, \emph{On complete intersections containing a linear subspace}, arXiv:1812.06682 (2018), 6 p.

\bibitem {Bea} A.~Beauville, \emph{Sous-vari\'et\'es sp\'eciales des vari\'et\'es de Prym}, Comp.\ Math.\ {\bf 45} (1982), 357--383.

\bibitem {Bea1} A.~Beauville, \emph{Quantum cohomology of complete intersections}, Mat.\ Fiz.\ Anal.\ Geom.\ {\bf 2} (1995), 384--398. 

\bibitem {BK} R.~Behesti, N.~Mohan Kumar, \emph{Spaces of rational curves in complete intersections}, Compos.\ Math.\ {\bf 149} (2013), 1041--1060.

\bibitem {BSD} E.~Bombieri, H.~P.~F.~Swinnerton-Dyer, \emph{On the local zeta function of a cubic threefold},
Annali della Scuola Normale Superiore di Pisa. Classe di Scienze, Ser.\ 3, {\bf 21} 
(1967), 1--29.

\bibitem {BH} L.~Bonavero, A.~Höring, \emph{Counting conics in complete intersections}, Acta Math.~Vietnam.\ {\bf 35} (2010), 23--30.

\bibitem {B} C.~Borcea, \emph{Deforming varieties of $k$--planes of projective complete intersections}, Pacific J.\ Math., {\bf 143} (1990), 25--36. 

\bibitem {Bo} R.~Bott, \emph{A residue formula for holomorphic vector-fields}, J.~Differential~Geom.\ {\bf 1} (1967), 311--330.

\bibitem {Br} M.~Brion, \emph{Equivariant cohomology and equivariant
intersection theory}, arXiv:math/9802063 (2008), 49 p.

\bibitem {Ca} C.~Casagrande,  \emph{Rank 2 quasiparabolic vector bundles on $\PP^1$ and the variety of linear subspaces contained in two odd-dimensional quadrics}, Math.\ Z.\ {\bf 280} (2015),  981--988.

\bibitem {CG} C.~H.~Clemens, Ph.~A.~Griffiths,
\emph{The Intermediate Jacobian of the Cubic Threefold}, 
Ann.\ of Math.\ {\bf 95} (1972),  281--356.

\bibitem {Co} A.~Collino, \emph{The Abel-Jacobi isomorphism for the cubic fivefold}, Pacific J.\ Math.\ {\bf 122} (1986), 43--55.

\bibitem {CV} J.~Cordovez, M.~ Valenzano, \emph{On Fano scheme of $k$-planes in a projective complete intersection}, 
Universit\'a di Torino, Quaderni
del Dipartimento di Matematica {\bf 21} (2005), 12p.

\bibitem {DM} O.~Debarre, L. Manivel, \emph{Sur la vari\'et\'e des espaces lin\'eaires contenus dans une intersection compl\'ete}, Math.\ Ann., {\bf 312} (1998), 549--574.

\bibitem {DR} U.~V.~Desale, S.~Ramanan, \emph{Classification of Vector Bundles of Rank $2$ on Hyperelliptic Curves}, lnvent.\ Math.\ {\bf 38} (1976), 161--185. 

\bibitem {Do} R.~Donagi, \emph{Group law on the intersection of two quadrics}, Ann.\ Scuola Norm.\ Sup.\ Pisa Cl.\ Sci.\ (4) {\bf 7} (1980),  217--239.

\bibitem {EG} D.~Edidin, W.~Graham, \emph{Localization in equivariant intersection theory and the Bott residue formula}, Amer.\ J.\ Math.\ {\bf 120} (1998), 619--636.

\bibitem {EH} D.~Eisenbud, J. Harris, \emph{3264 and All That. Intersection Theory in Algebraic Geometry},  Cambridge University Press, 2016.

\bibitem {F} G.~Fano, \emph{Sul sistema $\infty^2$ di rette contenuto in una variet\`a cubica generate dello spazio a quattro dimensioni},  Atti R.\ Acc.\ Sc.\ Torino {\bf XXXIX} (1904), 778--792.

\bibitem {Ful} W.~Fulton, \emph{Introduction to intersection theory in algebraic geometry}. Amer.\ Math.\ Soc.\ Providence, Rhode Island, 1980.

\bibitem {Fur} K.~Furukawa, \emph{Rational curves on hypersurfaces}, J.\ Reine Angew.\ Math.\ {\bf 665} (2012), 157--188.

\bibitem {Gh} F.~Gherardelli, \emph {Un osservazione sullia varieta cubica di $\PP^4$},
Rend.\ sem.\ mat.\ e fisicodi Milano {\bf 37} (1967), 157--160. 

\bibitem{HRS1} J.~Harris, M.~Roth, and J.~Starr, {\em Rational curves on hypersurfaces of low degree}, J.\ Reine Angew.\ Math.\ {\bf 571} (2004), 73--106.

\bibitem {H1} Hiep, Dang Tuan,
\emph{On the degree of the Fano schemes of linear subspaces on hypersurfaces}, Kodai Math.\ J.\ {\bf 39} (2016), 110--118. 

\bibitem {H2} Hiep, Dang Tuan,
\emph{Numerical invariants of Fano schemes of linear subspaces on complete intersections}, 	arXiv:1602.03659 (2017), 7 p. 

\bibitem {Is} V.~A.~Iskovskikh, \emph{Anticanonical models of three-dimensional algebraic varieties}, J.\ Soviet Math.\ {\bf 13} (1980), 745--814. Translated from: Itogi Nauki i Tekhniki.\ Ser.\ Sovrem.\ Probl.\ Mat.\ {\bf 12}, VINITI, Moscow, 1979, 59--157. 

\bibitem {ZJ} Zhi Jiang, \emph{A Noether-Lefschetz theorem for varieties of $r$-planes in complete intersections}, Nagoya Math.\ J.\ {\bf 206} (2012), 39--66.
%

\bibitem {Kol} J.~Koll\'ar, \emph{Rational Curves on Algebraic Varieties}, Ergebnisse der Mathematik und ihrer Grenzgebiete. 3. Folge, Band {\bf 32}. Springer, 1999. 

\bibitem {L} A.~Langer, \emph{Fano schemes of linear spaces on hypersurfaces}, Manuscripta Math.\ {\bf 93} (1997), 21--28. 

\bibitem {Li} A.~Libgober, \emph{Numerical characteristics of systems of straight lines on complete intersections}, Math.\ Notes {\bf 13} (1973), 51--56.
Translated from: Mat.\ Zametki {\bf 13} (1973), 87--96.

\bibitem{MaRoXaVa} J.~A.~Maia, A.~Rodrigues, F.~Xavier, and I.~Vainsencher, \emph{Enumeration of surfaces containing a curve of low degree}, preprint (2011), 19p.

\bibitem{Man0} L.~Manivel, \emph{Sur les hypersurfaces contenant des espaces lin\'eaires} [\emph{On hypersurfaces containing linear spaces}],
C.~R.~Acad.~Sci.~Paris S\'er. I Math.\ {\bf 328} (1999), 307--312.

\bibitem{Man} L.~Manivel, \emph{Symmetric functions, Schubert polynomials and degeneracy loci}. AMS Texts and Monographs {\bf 6}. Providence, RI, 2001.

\bibitem {Mar} D.~G.~Markushevich, \emph{Numerical invariants of families of lines on some Fano varieties}, Math.\ USSR-Sb.\ {\bf 44} (1983), 239--260. 
Translated from:  Matem.\ sb.\ {\bf 116(158)} (1981), 265--288.

\bibitem {MV} A.~L.~Meireles Ara\'ujo, I.~Vainsencher, \emph{Equivariant intersection theory and Bott's residue formula}, 16th School of Algebra, Part I (Portuguese) (Brasilia, 2000). Mat.\ Contemp.\ {\bf 20} (2001), 1--70.


\bibitem {Mi} C.~Miyazaki, \emph{Remarks on $r$-planes in complete intersections}, Tokyo J.\ Math.\ {\bf 39} (2016), 459--467.

\bibitem {Mo} U.~Morin, \emph{Sull’insieme degli spazi lineari contenuti in una ipersuperficie algebrica}, Atti Accad.\ Naz.\ Lincei Rend.\ Cl.\ Sci.\ Fis.\ Mat.\
Nat.\ {\bf  24} (1936), 188--190.

\bibitem {Pr} A.~Predonzan, \emph{Intorno agli $S_k$ giacenti sulla variet\`a intersezione completa 
di pi\`u forme},   Rend.\  Accad.\  Naz.\  dei  Lincei, (8), {\bf 5}  (1948),  238--242.

\bibitem {Re} M.~Reid, {\em The complete intersection of two or more quadrics}. PhD thesis, Trinity College, Cambridge, 1972.

\bibitem {Ru} X.~Roulleau, \emph{Elliptic curve configurations on Fano surfaces}, Manuscripta Math.\ 129 (2009), 381–399.

\bibitem {Te} B.~R.~Tennison,  \emph{On the quartic threefold}, Proc.\ London Math.\ Soc.\ {\bf 29} (1974), 714--734.

\bibitem {T1} A.~N.~Tyurin, \emph{On the Fano surface of a nonsingular cubic in $\PP^4$}, Math.\ USSR Izvestia {\bf 4} (1970),
1207--1214.

\bibitem {T2} A.~N.~Tyurin, \emph{The geometry of the Fano surface of a nonsingular cubic $F \subset \PP^4$ and
Torelli Theorems for Fano surfaces and cubics}, Math.\ USSR Izvestia {\bf 5} (1971), 517--546.

\bibitem {T3} A.~N.~Tyurin, \emph{Five lectures on
three-dimensional varieties},
Russian Math.\ Surveys {\bf 27} (1972), 1--53.

\bibitem {T4} A.~N.~Tyurin, \emph{On intersection of quadrics}, Russian Math.\ Surveys {\bf 30} (1975), 51--105.
Translated from: Itogi Nauki i Tekhniki.\ Ser.\ Sovrem.\ Probl.\ Mat.\ {\bf 12}, VINITI, Moscow, 1979,  5--57.   

\bibitem {vdW} B.~L.~van~der~Waerden, \emph{Zur algebraischen Geometrie 2. Die geraden Linien auf
den Hyperflächen des $\PP^n$}, Math.\ Ann.\ {\bf 108} (1933), 253--259.

\end{thebibliography}
\end{document}